\documentclass[11pt]{amsart}
\usepackage[numbers, square]{natbib}
\usepackage{amssymb}
\usepackage{amsmath}
\usepackage{amsfonts}
\usepackage{geometry}
\usepackage{graphicx}
\usepackage{amsthm}
\usepackage{hyperref}
\usepackage[dvipsnames]{xcolor}

\providecommand{\U}[1]{\protect\rule{.1in}{.1in}}
\theoremstyle{plain}

\newtheorem{lemma}{Lemma}

\newtheorem{remark}{Remark}

\newtheorem{theorem}{Theorem}
\numberwithin{equation}{section}
\newcommand{\lap}{\mbox{$\triangle$}}

\newcommand{\disp}{\displaystyle}

\newcommand{\iny}{\infty}

\newcommand{\rh}{\rho}
\newcommand{\ep}{\epsilon}

\newcommand{\abs}[1]{\left\vert#1\right\vert}

\newcommand{\pr}[1]{\left( #1 \right) }

\begin{document}
\title[Quantitative uniqueness of solutions to parabolic equations]
{Quantitative uniqueness of solutions to parabolic equations}
\author{ Jiuyi Zhu}
\address{
Department of Mathematics\\
Louisiana State University\\
Baton Rouge, LA 70803, USA\\
Email:  zhu@math.lsu.edu }
\thanks{Zhu is supported in part by  NSF grant DMS-1656845}
\date{}

\begin{abstract}
 We investigate the quantitative uniqueness of solutions
to parabolic equations with lower order terms on compact smooth manifolds. Quantitative uniqueness
 is a quantitative form of strong
unique continuation property. We characterize quantitative uniqueness by the rate of vanishing. We can obtain
the vanishing order of
solutions by $C^{1, 1}$ norm of the potential functions, as well as the $L^\infty$ norm of the coefficient functions. Some quantitative Carleman estimates and three cylinder inequalities are established.
\end{abstract}

\maketitle
\section{Introduction}

In this paper, we study the quantitative uniqueness for parabolic
 equations with non-trivial lower order terms on compact smooth manifolds.
Suppose $u$ is a
non-trivial solution to
\begin{equation}
\triangle_g u-\partial_t u -\tilde{V}(x, t) u=0 \quad \mbox{on} \ \mathcal{M}^1, \label{goal}
\end{equation}
 where $\mathcal{M}^1=\mathcal{M}\times (-1, 1)$ and $g$ is the metric on the compact smooth manifold $\mathcal{M}$ with dimension $n\geq 2$. Assume that $\tilde{V}\in C^{1, 1}$, where
 $\|\tilde{V}\|_{C^{1, 1}(\mathcal{M}^1)}=\sup_{\mathcal{M}^1}|\tilde{V}|+\sup_{\mathcal{M}^1}|\nabla \tilde{V}|+
 \sup_{\mathcal{M}^1}|\partial_t \tilde{V}|$. We may
 assume that $\|\tilde{V}\|_{C^{1, 1}(\mathcal{M}^1)}\leq M$ for $M\geq 1$.
Quantitative uniqueness, also called as quantitative unique continuation, described by the vanishing order,
characterizes how much the solution vanishes.
  It is a quantitative way to describe the strong unique
continuation property. If the condition that solution vanishes of infinite order at a point implies that the solution
vanishes identically, then we say the strong unique
continuation property holds.

Let's first review the progresses about quantitative uniqueness for elliptic equations. Recently, there are much attentions in this topic.
The most interesting example for quantitative unique continuation arises from the study of nodal sets for eigenfunctions on manifolds. For classical
eigenfunctions on a compact smooth Riemannian manifold
$\mathcal{M}$,
\begin{equation}-\lap_{g} \phi_\lambda=\lambda \phi_\lambda \quad \quad \mbox{in} \ \mathcal{M}.
 \label{eigen} \end{equation}
Donnelly and Fefferman in \cite{DF88} obtained that  the maximal
vanishing order of $\phi_\lambda$ is everywhere less than
$C\sqrt{\lambda}$, here $C$ only depends on the manifold
$\mathcal{M}$. Such vanishing order for eigenfunction
$\phi_\lambda$ is sharp,  which can be verified from spherical
harmonics.

Kukavica in \cite{Ku98} studied the quantitative unique continuation for
Schr\"{o}dinger equation
\begin{equation}
-\lap u+V(x)u=0. \label{schro}
\end{equation}
 If $\|V\|_{C^1}\leq K $ for some large constant $K>1$.  Kukavica showed that the upper bound of
vanishing order is less than $CK$. From Donnelly and Fefferman's work in the case $V(x)=-\lambda$,  this upper bound is not optimal.
Recently,  by
different methods, the sharp vanishing order for solutions of (\ref{schro})
is shown to be less than $C K^{\frac{1}{2}}$
independently by Bakri in \cite{Bak12} and Zhu in \cite{Zhu16}. It matches the optimal result for the vanishing order of eigenfunctions in Donnelly and
Fefferman's work in \cite{DF88}.

If $\|V\|_{L^\infty}\leq K_0$ for some large $K_0>1$, Bourgain and Kenig \cite{BK05} considered the vanishing order for (\ref{schro}) with the background from
Anderson localization for the Bernoulli model. Bourgain and Kenig
established that
\begin{equation}
\|u\|_{L^\infty (\mathbb B_r)}\geq c_1r^{c_2
K_0^{\frac{2}{3}}} \quad \quad \mbox{as}\ r\to 0,
\label{like}
\end{equation}
 where $c_1,
c_2$ depend only on $n$ and upper bound of the solution.  The estimates (\ref{like})
implies that the upper bound order of vanishing for solutions is less than
$CK_0^{\frac{2}{3}}.$ Moreover, Kenig in \cite{Ken07}
pointed out that the exponent $\frac{2}{3}$ of $K_0^{\frac{2}{3}}$
is optimal for complex valued potential function $V(x)$ based on Meshkov's example in
\cite{Mes92}. Especially, if the real valued potential function $V(x)\geq 0$, Kenig, Silvestre and Wang \cite{KSW15} were able to show that the vanishing order is less than
$CK_0^{\frac{1}{2}}$ in the planar domain.

For the general second order elliptic equation
\begin{equation} -\triangle u+ W(x)\cdot \nabla u+V(x) u=0  \quad
 \label{subgoal}
\end{equation}
with $\|W(x)\|_{L^\infty}\leq K_1,$
Bakri in \cite{Bak13} and  Davey in \cite{Dav14} independently  generalized the quantitative uniqueness result and obtained  that the order of vanishing is less than
$C(K_0^{2/3}+K_1^{2} )$.

The strong unique continuation property also holds for
second order elliptic equation (\ref{subgoal}) with singular lower
terms in $L^r$ Lebesgue space, i.e.
$$  W(x)\in L^{t} \ \mbox{with} \ t>n \quad \mbox{and} \quad  V(x)\in L^{\frac{n}{2}}. $$
Very Recently, Davey and the author in \cite{DZ17} established  a new quantitative $L^p\to L^q$ Carleman estimates
for a range of $p$ and $q$ value. We were
  able to deal with (\ref{subgoal}) with both singular
gradient potential $W(x)$ and singular potential $V(x)$ for $n\geq 3$.
Our results work for a large range of singular potentials $V(x)$
and gradient potentials $W(x)$ . Especially, for $n=2$, Davey and the author in \cite{DaZ17} were able to characterize
 vanishing order for all admissible singular potentials $V(x)\in L^s$ for $s>1$ and gradient potentials $W(x)\in L^t$ for $t>2$. It offers a complete description of quantitative unique continuation for second order elliptic equations in $n=2$.

Next, let's  briefly review some literature about strong unique continuation property for parabolic equations. We aim to study quantitative unique continuation for parabolic equations. The strong unique continuation property for parabolic equations with time-independent coefficients were shown by Landis and Oleink \cite{LO74}
and independently by Lin \cite{L88}.
The unique continuation property for parabolic equations with time-dependent coefficient was proved by e.g. Sogge \cite{S90}, Poon \cite{P96}, Chen \cite{C96},   Escauriaza, Fern\'andez and Vega in \cite{E00}, \cite{EV01}, \cite{EF03}, \cite{F03}, to just mention a few. In particular, Poon in \cite{P96} defined a suitable frequency function to measure the space-time vanishing rate. It was shown that if $u$ satisfies the inequality
$$|\triangle u+\partial_t u|\leq N_0 (|\nabla u|+|u|) \quad  \mbox{in}  \ \mathbb R^n\times [0 , \ T) $$ for some positive constant $N_0$
and $u$ vanishes to the infinite order in both space and time variable at $(0, 0)$, then $u$ is trivial in  $\mathbb R^n\times [0 , \ T)$. In term of the concept of vanishing of infinite order in both space and time, we mean for all $k\geq 1$, there exist $C_k$ such that
$$ |u(x, t)|\leq C_k (|x|+\sqrt{t})^k  $$
for all $(x, t)$ near $(0, 0)$.

In \cite{E00} and \cite{EV01}, Escauriaza and Vega  proved some Carleman inequalities and obtained strong unique continuation property for
global (defined in $\mathbb R^n\times [0, T)$ and local solutions for the parabolic equations
\begin{equation}
\triangle u+\partial_t u=V(x,t) u
\label{intro}
\end{equation}
for some unbounded potential $V(x, t)$. In particular, they showed that in certain $L^r_x L^s_t$ Lebesgue space for the potential function $V(x,t)$, the solution vanishes globally if the solution vanishes infinite order in the space-time variable at $(0,0)$.   Later, Koch and Tataru \cite{KT09} further proved this property for rough variable coefficients and rough $L^r_x L^s_t$ potentials.

In \cite {V02} and \cite{V03}, Vessella considered another interesting strong unique continuation property for parabolic equations (\ref{intro}) in $D\times (-T, T)$, where $D$ is a domain in $\mathbb R^n$. If $V(x,t )$ is a bounded function and $u$ vanishes at infinity order in the spacial variable as
\begin{equation}
\int_{\bar{Q}^T_r(x_0)}\ dx dt=O(r^N), \quad \mbox{as} \ r\to 0
\label{order}
\end{equation}
for every $N\in \mathbb N$, then $u$ vanishes in $D\times (-T, T)$. Here $\bar{Q}^T_r(x_0)= B_r(x_0)\times (-T, \ T)$ for $x_0\in D$ and $B_r(x_0)$ is a ball centered at $x_0$ with radius $r$ in $D$ in the Euclidean space. If $u$ is not trivial, from the strong unique continuation property,  then the condition (\ref{order}) will not hold for every $N$. A nature question is how large the possible $N$ is in (\ref{order}). That is, we aims to quantify this strong unique continuation property by studying the rate of vanishing.  If the strong unique
continuation property holds for the solutions and solutions do not
vanish of infinite order, the vanishing order of solutions depends
on the coefficient functions appeared in the
equations. So it is of interest to find out the relation between the vanishing order and
the potential function $V(x, t)$.

 Inspired by the progresses for the quantitative uniqueness of elliptic equations, it is interesting to study this topic for parabolic equations. We are interested in obtaining the vanishing order
characterized by the spatial variable in (\ref{order}).
 Since the $L^2$ norm and $L^\infty$ norm are comparable for second order parabolic equations, we define the vanishing order for the solution in (\ref{goal}) at $x_0\in \mathcal{M}$ by
\begin{equation}\sup\{ k| \ \ \limsup_{r\to 0^+}\frac{\sup_{Q^1_r(x_0)} |u|}{r^k}\},
\label{defi}
\end{equation}
where $Q_r^1(x_0)=\mathbb B_r(x_0)\times (-1, 1)$ in this context. $\mathbb B_r(x_0)$ is the geodesic ball with radius $r$ centered at $x_0$ on the manifold $\mathcal{M}$. $r=d(x, x_0)$ is the Riemanian distance from $x$ to
$x_0$.

Since the strong unique continuation property for (\ref{goal}) is shown, our goal is to consider the vanishing order of the solutions on $\mathcal{M}$. By the definition of the manifold $\mathcal{M}^1=\mathcal{M}\times (-1, \ 1)$,  we can write $\mathcal{M}\times {0}$ as $\mathcal{M}$. We work on finding out the estimates at $x_0$ on $\mathcal{M}$ in the form
\begin{equation}
\|u\|_{L^\infty( Q_r^1(x_0))}\geq Cr^N.
\end{equation}
From (\ref{defi}), it implies the vanishing order of solution $u$ at $x_0$ on $\mathcal{M}$ is less than $N$.

We may normalize the solutions $u$ in (\ref{goal}) as follow,
\begin{equation}
\|u(x,t)\|_{L^\infty(\mathcal{M})}\geq 1 \quad  \mbox{and}  \quad   \|u(x,t)\|_{L^\infty(\mathcal{M}^1)}\leq C_0.
\label{normalize}
\end{equation}
See also remark \ref{rem2} for additional information on the normalization of solutions  in (\ref{normalize2}).
If $\tilde{V}(x, t)\in C^{1, 1}$, using quantitative Carleman estimates, three cylinder inequalities and propagation of smallness argument, we are able to show the following theorem.
\begin{theorem}
The vanishing order of solutions to (\ref{goal}) on $\mathcal{M}$ is everywhere less than
$ CM^{\frac{1}{2}},$
where $C$ is a positive constant depending only on the manifold $\mathcal{M}$ and $C_0$.
\label{th1}
\end{theorem}

If we review the quantitative uniqueness result for elliptic equations (\ref{schro}), the $\frac{1}{2}$ exponent of the upper bound $M^{\frac{1}{2}}$ for the $C^{1,1}$ norm of potential function $\tilde{V}$ matches the one by \cite{Bak12} and \cite{Zhu16} for $C^1$ norm of $V$ in (\ref{schro}). It seems to be a sharp result. For example, if we can consider the case that $u(x, t)=\phi_\lambda(x)$ where $\phi_\lambda(x)$ is the eigenfunction in (\ref{eigen}), then $\tilde{V}(x, t)=-\lambda$ in equation (\ref{goal}). The statement of Theorem \ref{th1} agrees with the Donnelly and Fefferman's sharp results in \cite{DF88}.

We are also able to study the vanishing order for parabolic equations with non-trivial bounded lower order terms,
\begin{equation}
\triangle_g u-\partial_t u-W(x,t)\cdot \nabla u -{V}(x, t) u=0 \quad \mbox{on} \ \mathcal{M}^1, \label{goal2}
\end{equation}
where
\begin{equation}
\|W\|_{L^\infty(\mathcal{M}^1)}\leq M_1 \quad  \mbox{and} \quad  \|V\|_{L^\infty(\mathcal{M}^1)}\leq M_0,
\label{condit}
\end{equation}
with $M_0, M_1\geq 1$.
We are able to show that
\begin{theorem}
The vanishing order of solutions to (\ref{goal2}) on $\mathcal{M}$ is everywhere less than
\begin{equation} C(M_0^{\frac{2}{3}}+M_1^2), \label{similar}\end{equation}
where $C$ is a positive constant depending only on the manifold $\mathcal{M}$ and $C_0$.
\label{th2}
\end{theorem}

From Kenig's observation in \cite{Ken07}, the power $\frac{2}{3}$ for $M_0^{\frac{2}{3}}$ in the theorem seems to be optimal.
Very recently, Camliyurt and Kukuvica \cite{CK17} studied the quantitative unique continuation for the global solutions of (\ref{goal2}) in $\mathbb R^n\times (0 ,\ T)$. By assuming periodicity of solutions, similar upper bound of vanishing order as (\ref{similar}) for spatial and time variable was obtained. Parabolic frequency function and similarity variable argument were used in \cite{CK17}. Our arguments are relied on quantitative Carleman estimates and three cylinder inequalities.

Besides the important roles of quantitative uniqueness in size measurement of nodal sets \cite{DF88}, spectral theory of Schr\"odinger equations \cite{BK05}, backward uniqueness \cite{EF03}, it also finds applications in inverse problems and control theory \cite{AN08} and other topics, to just mention a few.

 The paper is organized as follows. In section 1, we prove the quantitative
Carleman estimates for second order parabolic equation with $C^{1,1}$ potentials or with $L^\infty$ potentials. Section
2 is devoted to the proof of three cylinder inequalities from Carleman estimates. In section 3, using the propagation of smallness argument, we show the proof of
Theorem \ref{th1} and Theorem \ref{th2}. In the paper, since we are interested in the dependence of vanishing order on $M$, $M_1$ and $M_2$, we assume that they are large constants.
The letters $c$, $C$, $C_1$ and $C_2$  denote generic positive constants that do not depends on $u$, and may vary from line to line.

\section{Carleman estimates}

In this section, we show the quantitative Carleman estimates for parabolic equations. We drop the notation of metric $g$ and simply write $\triangle$ for the Laplace-Beltrami  operator $\triangle_g$. Carleman estimates are weighted integral inequalities with a weight function
$e^{-\tau g(r)}$, where the function $g(r)$ usually satisfies some convexity properties. Let's define the weight function. For a fixed number $\ep$ such that $0<\ep<1$ and $\rh_0<0$, we define $f$ on $(-\infty, \ \rh_0)$ by  $f(\rh)=\rh+e^{\ep \rh}$.  We introduce the weight function $$g(r)=f^{-1}(\ln r)$$ for small $r$. We can check that $g(r)\approx \ln r$ as $r\to 0$. Now we state the main results in this section.

\begin{theorem}
There exist positive constants $r_0$, $C$, $C_1$ and $C_2$, which depend only on $\mathcal{M}$ and $\ep$, such that, for any $\tilde{V}\in C^{1, 1}(\mathcal{M}^1)$, $x_0\in \mathcal{M}$, $u\in C^{\infty}_0\big (Q^T_{r_0}(x_0) \backslash \big\{\{ x_0\}\times(-T , \ T)\big\}\big)$ and $\tau>C(1+ \|\tilde{V}\|_{C^{1, 1}}^{\frac{1}{2}})$, one has
\begin{align}
\|(\triangle u-\partial_t u-\tilde{V}u)e^{-\tau g(r)}r^{\frac{4-n}{2}}\|_{L^2}& \geq C_1\tau^{\frac{1}{2}} \| \nabla u e^{(-\tau+\frac{\ep}{2})g(r)} r^{\frac{2-n}{2}}\|_{L^2} \nonumber\\ &+C_2\tau^{\frac{3}{2}} \| u e^{(-\tau+\frac{\ep}{2})g(r)} r^{-\frac{n}{2}}\|_{L^2}.
\label{Carle}
\end{align}
\label{thCarle}
\end{theorem}

As a consequence, we have the following Carleman estimates which do not involve potential functions.
\begin{lemma}
There exist positive constants $r_0$, $C$, $C_1$ and $C_2$,  which depend only on $\mathcal{M}$ and $\ep$, such that, for any  $x_0\in \mathcal{M}$, $u\in C^{\infty}_0\big (Q^T_{r_0}(x_0) \backslash \big\{\{ x_0\}\times(-T , \ T)\big\}\big)$  and $\tau>C$, one has
\begin{align}
\|(\triangle u-\partial_t u )e^{-\tau g(r)}r^{\frac{4-n}{2}}\|_{L^2}& \geq C_1\tau^{\frac{1}{2}} \| \nabla u e^{(-\tau+\frac{\ep}{2})g(r)} r^{\frac{2-n}{2}}\|_{L^2} \nonumber\\ &+C_2\tau^{\frac{3}{2}} \| u e^{(-\tau+\frac{\ep}{2})g(r)} r^{-\frac{n}{2}}\|_{L^2}.
\label{Carle1}
\end{align}
\label{le1}
\end{lemma}

If we set $\tilde{V}(x, t)=0$, then the Carleman estimates (\ref{Carle1}) in Lemma \ref{le1} follows from (\ref{Carle}). It is also obtained by Vessella in \cite{V03}.

To deal with the equation with (\ref{goal2}) with bounded coefficient functions, we need to establish the following Carleman estimates.
\begin{theorem}
There exist positive constants $r_0$, $C$, $C_1$ and $C_2$, which depend only on $\mathcal{M}$ and $\ep$, such that, for any ${V}, {W} \in L^\infty(\mathcal{M}^1)$, $x_0\in \mathcal{M}$, $u\in C^{\infty}_0\big (Q^T_{r_0}(x_0) \backslash \big\{\{ x_0\}\times(-T , \ T)\big\}\big)$  and $\tau>C(1+ \|{V}\|_{L^\infty}^{\frac{2}{3}}+\|{W}\|_{L^\infty}^2) $, one has
\begin{align}
\|(\triangle u -\partial_t u -W\cdot \nabla u-{V}u)e^{-\tau g(r)}r^{\frac{4-n}{2}}\|_{L^2}& \geq C_1\tau^{\frac{1}{2}} \| \nabla u e^{(-\tau+\frac{\ep}{2})g(r)} r^{\frac{2-n}{2}}\|_{L^2} \nonumber\\ &+C_2\tau^{\frac{3}{2}} \| u e^{(-\tau+\frac{\ep}{2})g(r)} r^{-\frac{n}{2}}\|_{L^2}.
\label{Carle2}
\end{align}
\label{th4}
\end{theorem}

We first show the proof of Theorem \ref{th4} from Lemma \ref{le1}.
\begin{proof}
By the triangle inequality, it follows that
\begin{align}
\|(\triangle u-\partial_t u-W\cdot \nabla u-{V}u)e^{-\tau g(r)}r^{\frac{4-n}{2}}\|_{L^2}& \geq \|(\triangle u-\partial_t u)e^{-\tau g(r)}r^{\frac{4-n}{2}}\|_{L^2} \nonumber \\&  -\|V\|_{L^\infty}\|ue^{-\tau g(r)}r^{\frac{4-n}{2}}\|_{L^2}\nonumber \\&-\|W\|_{L^\infty}\|\nabla ue^{-\tau g(r)}r^{\frac{2-n}{2}}\|_{L^2}.
\label{triangle}
\end{align}
By the assumption of $\tau$ in the theorem, we choose $C$ in the lower bound of $\tau$ such that
\begin{equation}
\|V\|_{L^\infty}\|ue^{-\tau g(r)}r^{\frac{4-n}{2}}\|_{L^2}\leq \frac{C_1}{2}\tau^{\frac{3}{2}} \| u e^{(-\tau+\frac{\ep}{2})g(r)} r^{-\frac{n}{2}}\|_{L^2}
\label{fuc1}
\end{equation}
and
\begin{equation}
\|W\|_{L^\infty}\|\nabla ue^{-\tau g(r)}r^{\frac{4-n}{2}}\|_{L^2}\leq \frac{C_2}{2} \tau^{\frac{1}{2}} \| \nabla u e^{(-\tau+\frac{\ep}{2})g(r)} r^{\frac{2-n}{2}}\|_{L^2},
\label{fuc2}
\end{equation}
where $C_1$ and $C_2$ are those appeared in (\ref{Carle1}). Applying (\ref{Carle1}) in the inequality (\ref{triangle}) and using (\ref{fuc1}) and (\ref{fuc2}), we arrive at (\ref{Carle2}).
\end{proof}

The rest of the section is devoted to the proof of Theorem \ref{thCarle}. We adapt the strategy from the proof of (\ref{Carle1}) in \cite{V03}.
\begin{proof}[Proof of Theorem \ref{thCarle}]

We use polar geodesic coordinate $(r, \theta)$ near $x_0$. Using the Einstein's notation,
$$\triangle u -\partial_t u=\partial^2_r u+(\partial_r \ln \sqrt{\gamma}+\frac{n-1}{r})\partial_r u+\frac{1}{r^2}\triangle_{\theta} u-\partial_t u,  $$
where $$\triangle_\theta=\frac{1}{\sqrt{\gamma}}\partial_i(\sqrt{\gamma}\gamma^{ij}\partial_j u),$$ $\partial_i=\frac{\partial}{\partial \theta_i}$ and $\gamma_{ij}(r, \theta)$ is the metric on $S^{n-1}$. We write $\gamma=\det(\gamma_{ij})$. Since $\mathcal{M}$ is a compact smooth manifold,
it is well known that
\begin{equation}
\left\{\begin{array}{lll}
C^{-1}\leq \gamma \leq C;  \medskip\\
\partial_r (\gamma^{ij})\leq C(\gamma^{ij})\quad (\mbox{in the sense of tensors});  \medskip\\
|\partial_r(\gamma)|\leq C
\end{array}
\right.
\label{gama}
\end{equation}
for small enough $r$.
Set a new coordinate $z=\ln r$. In the new coordinates, it follows that
\begin{equation}\triangle u -\partial_t  u= e^{-2z}\big( \partial^2_{z} u+(\partial_{z} \ln \sqrt{\gamma}+(n-2))\partial_{z} u+\triangle_{\theta} u\big)-\partial_t u.  \label{star} \end{equation}
Then we introduce a new transformation $z=f(\rh)$ with $f(\rh)=\rh+ e^{\epsilon \rh}$ for some fixed $0<\epsilon<1$. Under this transformation, (\ref{star}) will take the following expression
\begin{align*}\triangle u -\partial_t u&= e^{-2f(\rh)}(1+\ep e^{\ep \rh})^{-2}
\Big[ \partial^2_{\rh} u+ \big((n-2)(1+\ep e^{\ep \rh})-\frac{ \ep^2 e^{\ep \rh}}{ 1+\ep e^{\ep \rh}}\big)\partial_\rh u+ \partial_\rh(\ln \sqrt{\gamma})\partial_\rh u
\\ &+(1+\ep e^{\ep \rh})^2\triangle_{\theta} u\Big]-\partial_t u.
\end{align*}
Let
\begin{equation} \mathcal{L} u=\triangle u -\partial_t u-\tilde{V}(x, t)u.
\end{equation}
Due to those changes of variables, the function $u$ is in the variable $(\rh, \theta, t)$. The operator $\mathcal{L}$ takes the form
\begin{equation}
\mathcal{L} u=e^{-2f(\rh)}(1+\ep e^{\ep \rh})^{-2}(\mathcal{Q}(u)+\tilde{\mathcal{Q}} (u))
\label{LLL}
\end{equation}
where
\begin{align*}\mathcal{Q}(u)=&\partial^2_{\rh} u+ \big((n-2)(1+\ep e^{\ep \rh})-\frac{ \ep^2 e^{\ep \rh}}{ 1+\ep e^{\ep \rh}}\big)\partial_\rh u+ (1+\ep e^{\ep \rh})^2\triangle_{\theta} u
\\ & - e^{2f(\rh)}(1+\ep e^{\ep \rh})^{2} \partial_t u
-e^{2f(\rh)}(1+\ep e^{\ep \rh})^{2} \tilde{V}(e^{f(\rh)}, \theta, t)u
\end{align*}
and
$$\tilde{\mathcal{Q}} (u)=\partial_\rh (\ln\sqrt{\gamma})\partial_\rh u.   $$
For the ease of notation, let
\begin{equation}  a(\rh)=e^{2f(\rh)}(1+\ep e^{\ep})^2
\label{aaa}
\end{equation}
and
\begin{equation} b(\rh)=(n-2)(1+\ep e^{\ep \rh})-\frac{ \ep^2 e^{\ep \rh}}{ 1+\ep e^{\ep \rh}}.
\label{bbb}
\end{equation}
Then
\begin{align*}\mathcal{Q}(u)=&\partial^2_{\rh} u+ b(\rh)\partial_\rh u+ (1+\ep e^{\ep \rh})^2\triangle_{\theta} u
 - a(\rh) \partial_t u
+a(\rh) \tilde{V}(e^{f(\rh)}, \theta, t)u.
\end{align*}

Since $u\in C^{\infty}_0\big (Q^T_{r_0}(x_0) \backslash \big\{\{ x_0\}\times(-T , \ T)\big\}\big)$,  in term of the variable $(\rh, \theta, t)$,  the function $u$ has support in  $(-\infty, \rh_0)\times S^{n-1}\times (-T, T)$, where $\rh_0=-|\rh_0|$ with $|\rh_0|$ chosen to be sufficiently large.
Set $$u= e^{\tau \rh} v.$$
We introduce a conjugate operator
$$\mathcal{Q}_\tau (v)=e^{-\tau \rh}\mathcal{Q}(u)=e^{-\tau \rh}\mathcal{Q}(e^{\tau \rh}v).$$
Direct computations show that
\begin{equation} \mathcal{Q}_\tau (v)=\mathcal{Q}_\tau^1 (v)+ \mathcal{Q}_\tau^2 (v), \label{QQ1}\end{equation}
where
\begin{equation}\mathcal{Q}_\tau^1 (v)=\partial_\rh^2 v+( b(\rh)\tau+\tau^2)v-a(\rh) \tilde{V}(e^{f(\rh)}, \theta, t)v+ (1+\ep e^{\ep \rh})^2\triangle_{\theta} v  \label{QQ2} \end{equation}
and
\begin{equation} \mathcal{Q}_\tau^2 (v)=(2\tau+b(\rh))\partial_\rh v- a(\rh)\partial_t v.  \label{QQ3}\end{equation}
Furthermore, set
\begin{equation} A_0(\rh)=\tau^2+b(\rh)\tau
\label{A0A0}
\end{equation}
and
\begin{equation} A_1(\rh)=2 \tau+b(\rh).
\label{A1A1}
\end{equation}
Then $\mathcal{Q}_\tau^1 (v)$ and $\mathcal{Q}_\tau^2 (v)$ can be rewritten as follows,
\begin{equation}
\mathcal{Q}_\tau^1 (v)=\partial_\rh^2 v+A_0(\rh) v-a(\rh) \tilde{V}(e^{f(\rh)}, \theta, t)v+ (1+\ep e^{\ep \rh})^2\triangle_{\theta} v \label{back} \end{equation}
and
\begin{equation}
\mathcal{Q}_\tau^2 (v)= A_1(\rh) \partial_\rh v- a(\rh)\partial_t v.
\label{back1}\end{equation}

To deal with the integration on $(-\infty, \rh_0)\times S^{n-1}\times (-T, T)$, we introduce the $L^2$ norm as
$$ \|v\|^2=\int_{(-\infty, \rh_0)\times S^{n-1}\times (-T, T)} v^2 \sqrt{\gamma} d\rh d\theta dt, $$
where $d \theta$ is the measure on $S^{n-1}$.
From (\ref{LLL}), we obtain that
\begin{align}
\| e^{-\tau \rh} e^{2f(\rh)} (1+\ep e^{\ep \rh})^2 \mathcal{L} u\|&= \|\mathcal{Q}_\tau (v)+\tilde{{\mathcal{Q}}}_\tau (v)\| \nonumber\\
&\geq \|\mathcal{Q}_\tau (v)\|-\|\tilde{{\mathcal{Q}}}_\tau (v)\|,
\label{basic}
\end{align}
where
\begin{align*}
\tilde{{\mathcal{Q}}}(v)&= e^{-\tau\rh} \tilde{{\mathcal{Q}}}(u)= e^{-\tau\rh} \tilde{{\mathcal{Q}}}(e^{\tau \rh} v).
\end{align*}

From (\ref{gama}) and the definition of $f(\rh)$, it follows that
\begin{align}
|\partial_\rh (\ln \sqrt{\gamma})|&=|\frac{\gamma'}{2\gamma}e^{f(\rh)}f'(\rh)|   \nonumber \\
&\leq C e^{\rh}. \label{often}
\end{align}
Furthermore, it implies that
\begin{align}
|\tilde{{\mathcal{Q}}}(v)|&=e^{-\tau\rh} |\partial_\rh (\ln \sqrt{\gamma})\partial_\rh (e^{\tau \rh} v)|\nonumber \\
&\leq C e^{\rh}|\tau v+\partial_\rh v|.
\label{below}
\end{align}
Later on,  we will show that $\|\tilde{{\mathcal{Q}}}(v)\|$ can be controlled by $\|\mathcal{Q}_\tau (v)\|$.

Now we focus on the estimates on $\mathcal{Q}_\tau (v)$. Squaring $\mathcal{Q}_\tau (v)$ in (\ref{QQ1}) gives that
\begin{equation}
\|\mathcal{Q}_\tau(v)\|^2=\|\mathcal{Q}^1_\tau (v)\|^2+\|\mathcal{Q}^2_\tau (v)\|^2+2<\mathcal{Q}^1_\tau (v), \mathcal{Q}^2_\tau (v)>.
\label{imporr}
\end{equation}
We study each other term in the right hand side of (\ref{imporr}). We first consider the inner product $<\mathcal{Q}^1_\tau (v), \mathcal{Q}^2_\tau (v)>$. From the expression of $\mathcal{Q}^1_\tau (v)$ in (\ref{back}) and $\mathcal{Q}^2_\tau (v)$ in (\ref{back1}), we have
\begin{align}
2<\mathcal{Q}^1_\tau (v), \mathcal{Q}^2_\tau (v)>&= 2\int \partial^2_\rh v\big(A_1(\rh)\partial_\rh v-a(\rh)\partial_t v\big)\sqrt{\gamma} \nonumber \\
&+2\int A_0(\rh)v\big(A_1(\rh)\partial_\rh v-a(\rh)\partial_t v\big)\sqrt{\gamma} \nonumber \\
&+ 2\int (1+\ep e^{\ep \rh})^2\triangle_\theta v\big(A_1(\rh)\partial_\rh v-a(\rh)\partial_t v\big)\sqrt{\gamma} \nonumber \\
&-2 \int a(\rh)\tilde{V}(e^{f(\rh)}, \theta, t) v \big(A_1(\rh)\partial_\rh v-a(\rh)\partial_t v\big)\sqrt{\gamma}.
\label{right}
\end{align}
Next we compute each term in the right hand side of the last equality. Integration by parts shows that
\begin{align}
2\int \partial^2_\rh v A_1(\rh)\partial_\rh v \sqrt{\gamma}&= \int A_1(\rh)\partial_\rh (\partial_\rh v)^2 \sqrt{\gamma}  \nonumber \\
&=-\int A_1^{'}(\rh)|\partial_\rh v|^2\sqrt{\gamma}-\int A_1(\rh)|\partial_\rh v|^2 \partial_\rh(\ln \sqrt{\gamma})\sqrt{\gamma}.
\label{111}
\end{align}

From the integration by parts argument, we have
\begin{align}
-2\int \partial^2_\rh v a(\rh)\partial_t v \sqrt{\gamma}&= 2\int a'(\rh)\partial_\rh v\partial_t v \sqrt{\gamma}
+2\int a(\rh)\partial_\rh v\partial_t v \partial_\rh(\ln \sqrt{\gamma})\sqrt{\gamma} \nonumber \\
&+\int a(\rh)\partial_t(\partial_\rh v)^2 \sqrt{\gamma}\nonumber \\
&= 2\int a'(\rh)\partial_\rh v\partial_t v \sqrt{\gamma}
+2\int a(\rh)\partial_\rh v\partial_t v \partial_\rh(\ln \sqrt{\gamma})\sqrt{\gamma}
\label{222}
\end{align}
since $a(\rh)$ and $\sqrt{\gamma}$ do not depend on $t$.

Performing the integration by parts again yields that
\begin{align}
2\int A_0(\rh)v A_1(\rh) \partial_\rh v \sqrt{\gamma}&= \int A_0(\rh) A_1(\rh) \partial_\rh v^2 \sqrt{\gamma} \nonumber \\
&=\int \big(A_0(\rh) A_1(\rh)\big)_\rh v^2\sqrt{\gamma} -\int A_0(\rh) A_1(\rh) v^2 \partial_\rh(\ln \sqrt{\gamma})\sqrt{\gamma}. \label{222}
\end{align}

Since $A_0(\rh)$, $a(\rh)$ and $\sqrt{\gamma}$ are independent of $t$, the following term vanishes.
\begin{align}
-2\int A_0(\rh)v a(\rh) \partial_t v \sqrt{\gamma}& =-\int A_0(\rh)a(\rh)\partial_t v^2 \sqrt{\gamma} \nonumber \\
& =\int \big(A_0(\rh)a(\rh)\sqrt{\gamma}\big)_t v^2  \nonumber \\ &=0. \label{333}
\end{align}
We continue to investigate the right hand side of (\ref{right}). Note that $ |\nabla_\theta v|^2= \gamma^{ij}\partial_i u\partial_j u. $
Integration by parts yields that
\begin{align}
2\int (1+\ep e^{\ep \rh})^2\triangle_\theta v A_1(\rh)  \partial_\rh v  \sqrt{\gamma}&= -2\int (1+\ep e^{\ep \rh})^2 A_1(\rh)
\partial_\rh(\partial_i v)\partial_j v \gamma^{ij} \sqrt{\gamma} \nonumber \\
&=-\int (1+\ep e^{\ep \rh})^2 \partial_\rh |\nabla_\theta v|^2 A_1(\rh)\sqrt{\gamma}  \nonumber \\
&=\int \big( (1+\ep e^{\ep \rh})^2 A_1(\rh)\big)_\rh |\nabla_\theta v|^2 \sqrt{\gamma} \nonumber \\ &+\int (1+\ep e^{\ep \rh})^2 A_1(\rh)
|\nabla_\theta v|^2 \partial_\rh(\ln \sqrt{\gamma})\sqrt{\gamma}  \nonumber \\
&\geq   \int \big( (1+\ep e^{\ep \rh})^2 A_1(\rh)\big)_\rh |\nabla_\theta v|^2 \sqrt{\gamma}\nonumber \\ &-C\int (1+\ep e^{\ep \rh})^2 A_1(\rh)
|\nabla_\theta v|^2 e^{\rh} \sqrt{\gamma},
\label{444}
\end{align}
where we have used the estimates (\ref{often}) in the the last inequality. Recall the definition of $A_1(\rh)$ in (\ref{A1A1}).
We have
\begin{equation}
\big( (1+\ep e^{\ep \rh})^2 A_1(\rh)\big)_\rh \geq C \ep^2 \tau e^{\ep \rh}
\end{equation}
for some fixed $ 0<\ep<1$.
Since $-\infty<\rh<-|\rh_0|$ with $|\rh_0|$ large enough, from (\ref{444}), it follows that
\begin{align}
2\int (1+\ep e^{\ep \rh})^2\triangle_\theta v A_1(\rh)  \partial_\rh v  \sqrt{\gamma} \geq C \ep^2 \tau \int  e^{\ep \rh}|\nabla_\theta v|^2 \sqrt{\gamma}.
\label{555}
\end{align}

Since $a(\rh)$ and $\sqrt{\gamma}$ are independent of $t$, performing the integration by parts implies that
\begin{align}
-2\int  (1+\ep e^{\ep \rh})^2\triangle_\theta v a(\rh) \partial_t v \sqrt{\gamma}& = 2 \int (1+\ep e^{\ep \rh})^2 a(\rh)\partial_{t i}v
\gamma^{ij}\partial_j v \sqrt{\gamma} \nonumber\\
&=\int (1+\ep e^{\ep \rh})^2 a(\rh) \partial_t |\nabla_\theta v|^2\sqrt{\gamma} \nonumber \\
&=0.
\label{666}
\end{align}

Before calculating the integral involving $\tilde{V}$, we estimate the derivative of $a(\rh)$ and $A_1(\rh)$.
Recall that the definition of $a(\rh)$ in (\ref{aaa}), $A_1(\rh)$ in (\ref{A1A1}) and $f(\rh)=\rh+e^{\ep \rh}$. performing the derivative gives that
\begin{equation}
|a'(\rh)|\leq C e^{2\rh} \quad \mbox{and} \quad |A_1^{'}(\rh)|\leq C\ep^2 e^{\ep \rh}
\label{estima}
\end{equation}
for some fixed $\ep$ and large enough $|\rh_0|$.
Using integration by parts, we consider the last term in the right hand side of (\ref{right}), we obtain that
\begin{align}
-2\int a(\rh) \tilde{V}( e^{f(\rh)}, \theta, t) v A_1(\rh)\partial_\rh v \sqrt{\gamma}&= -\int a(\rh) \tilde{V}( e^{f(\rh)}, \theta, t) A_1(\rh) \partial_\rh v^2 \sqrt{\gamma} \nonumber \\
&=\int \{a'(\rh) \tilde{V} A_1(\rh)+a(\rh)\tilde{V}_r e^{f(\rh)}f'(\rh) A_1(\rh) \nonumber \\
&+a(\rh)\tilde{V}A_1^{'}(\rh)+ a(\rh) \tilde{V} A_1(\rh)\partial_\rh ( \ln \sqrt{\gamma})\big\} v^2 \sqrt{\gamma}.
\end{align}
From the assumption of $a(\rh)$ in (\ref{aaa}) and the estimates (\ref{estima}),
we get
\begin{align}
-2\int a(\rh) \tilde{V}( e^{f(\rh)}, \theta, t) v A_1(\rh)\partial_\rh v \sqrt{\gamma}\geq -C\tau \|\tilde{V}\|_{C^{1,1}} \int v^2 e^\rho \sqrt{\gamma}.
\label{777}
\end{align}
Similar arguments yield that
\begin{align}
2\int a(\rh) \tilde{V}( e^{f(\rh)}, \theta, t) v a(\rh)\partial_t v \sqrt{\gamma}&= \int a^2(\rh) \tilde{V}( e^{f(\rh)}, \theta, t) \partial_t v^2\sqrt{\gamma} \nonumber \\
&=-\int a^2(\rh) \tilde{V}_t( e^{f(\rh)}, \theta, t) v^2\sqrt{\gamma} \nonumber \\
&\geq - C\|\tilde{V}\|_{C^{1,1}} \int v^2 e^\rho \sqrt{\gamma}.
\label{888}
\end{align}
We have computed all the integrals in the right hand side of (\ref{right}). Combining all the terms in calculations from (\ref{111}) to (\ref{888}), we obtain
\begin{align}
\|\mathcal{Q}_\tau(v)\|^2=\|\mathcal{Q}_\tau^1(v)\|^2+I_1+I_2,
\label{compren}
\end{align}
where
\begin{align}
I_1&=\|\mathcal{Q}_\tau^2(v)\|^2-\int A_1^{'}(\rh)|\partial_\rh v|^2\sqrt{\gamma}+2\int a'(\rh)\partial_\rh v\partial_t v \sqrt{\gamma}\nonumber \\
&-\int \big( A_1(\rh)A_2(\rh)\big)_\rh v^2 \sqrt{\gamma}+C\ep^2\tau \int e^{\ep \rh} |\nabla_\theta v|^2 \sqrt{\gamma}
\label{I1I1}
\end{align}
and
\begin{align}
I_2&=-\int A_1(\rh) |\partial_\rh v|^2 (\partial_\rh \ln \sqrt{\gamma})\sqrt{\gamma}+2 \int a(z)\partial_t v \partial_\rh v (\partial_\rh \ln \sqrt{\gamma})\sqrt{\gamma}
 \nonumber \\&-\int A_0(\rh)A_1(\rh) v^2 (\partial_\rh \ln \sqrt{\gamma})\sqrt{\gamma}- C\tau \|\tilde{V}\|_{C^{1,1}} \int v^2 e^\rho \sqrt{\gamma}.
 \label{I2I2}
\end{align}
Direct calculations show that
$$A_0(\rh)A_1(\rh)= 2\tau^3+3b(\rh)\tau^2+b^2(\tau) \tau.  $$
Recall the definition of $ b(\rh)$ in (\ref{bbb}). Calculating the derivative gives that
\begin{align}
|\big(A_0(\rh)A_1(\rh)\big)_\rh|&= |3b'(\rh) \tau^2+2b(\rh)b'(\rh)\tau| \nonumber\\
&\leq C\tau^2 e^{\ep \rh}.
\label{inee}
\end{align}

Now we work on the expression $I_1$ to find a low bound. We estimate $I_1$ by
\begin{align}
I_1\geq J_1-C \int \tau^2 e^{\ep \rh} v^2 \sqrt{\gamma}+C\ep^2\tau \int e^{\ep \rh}|\nabla_\theta v|^2 \sqrt{\gamma},
\label{decom}
\end{align}
where
$$J_1= \|\mathcal{Q}_\tau^2(v)\|^2-\int A_1^{'}(\rh)|\partial_\rh v|^2\sqrt{\gamma}+2\int a'(\rh)\partial_\rh v\partial_t v \sqrt{\gamma}  $$
and we have used the estimate (\ref{inee}).
Recall the definition of $\mathcal{Q}_\tau^2$ in (\ref{back1}), we rewrite $J_1$ as
\begin{equation}
J_1= \int \{ |A_1(\rh)\partial_\rh v-a(\rh)\partial_t v|^2 -A_1^{'}(\rh)|\partial_\rh v|^2+ 2a'(\rh) \partial_t v \partial_\rh v\} \sqrt{\gamma}.
\end{equation}

To estimate $J_1$, let
 $$\alpha=\partial_\rh v \quad \mbox{and} \quad \beta=\partial_t v.$$
Introduce the expression $R(\rh, \tau; \alpha, \beta)$ following from \cite{V02} and \cite{V03} as
\begin{equation}
R(\rh, \tau; \alpha, \beta)=|A_1(\rh)\alpha-a(\rh)\beta|^2- A_1^{'}(\rh)\alpha^2+2 a'(\rh)\alpha\beta.
\label{formu}
\end{equation}

We claim that
\begin{equation}
 R(\rh, \tau; \alpha, \beta)\geq \tau |\alpha|^2+\frac{e^{5\rh}}{2\tau} |\beta|^2.
\end{equation}
From the definition of $a(\rh)$ in (\ref{aaa}), we obtain that
\begin{equation}
\frac{a'(\rh)}{a(\rh)}=\frac{2\ep^2+2(1+\ep e^{\ep \rh})f'(\rh)}{1+\ep e^{\ep \rh}}.
\label{denot}
\end{equation}
Since $f'(\rh)=1+\ep e^{\ep \rh}$, it can be shown that
\begin{equation}
|\frac{a'(\rh)}{a(\rh)}-2|\leq 4\ep.
\label{cut}
\end{equation}
On one hand, we reorganize (\ref{formu}) as
\begin{equation}
R(\rh, \tau; \alpha, \beta)=[(A_1(\rh)-\frac{a'(\rh)}{a(\rh)})\alpha-a(\rh)\beta]^2+\alpha^2[A_1^2(\rh)-A_1^{'}(\rh)-(A_1(\rh)-\frac{a'(\rh)}{a(\rh)})^2].
\end{equation}
Recall the definition of $A_1(\rh)$ in (\ref{A1A1}), using the estimates (\ref{denot}) and (\ref{cut}), we arrive at
\begin{equation}
R(\rh, \tau; \alpha, \beta)\geq 4\tau\alpha^2
\label{def1}
\end{equation}
for every $\alpha, \beta \in \mathbb R^2$, $\rh<-C$ and $\tau>C$.

On the other hand, we rewrite $R(\rh, \tau; \alpha, \beta)$ as
\begin{equation}
R(\rh, \tau; \alpha, \beta)=[A_1(\rh)\alpha-(a(\rh)- \frac{a'(\rh)}{A_1(\rh)})\beta]^2+ \frac{a'(\rh) a(\rh)}{A_1(\rh)}(2-\frac{a'(\rh) }{A_1(\rh)a(\rh)})\beta^2- A_1^{'}(\rh)\alpha^2.
\end{equation}
Using the definition of $A_1(\rh)$ in (\ref{A1A1}) and $a(\rh)$ in (\ref{aaa}), taking  the estimates (\ref{denot}) and (\ref{cut}) into considerations gives that
\begin{equation}
R(\rh, \tau; \alpha, \beta)\geq \frac{e^{5\rh}}{\tau}\beta^2-C\ep^2 \alpha^2
\label{def2}
\end{equation}
for every $\alpha, \beta \in \mathbb R^2$, $\rh<-C$ and $\tau>C$.
Combining the estimates (\ref{def1}) and (\ref{def2}), we have shown the claim.
That is, we have arrived at the estimates
\begin{equation}
\int R(\rh, \tau; v_\rh, v_t)\sqrt{\gamma}\geq  C\tau \int |v_\rh|^2 \sqrt{\gamma}+ \frac{1}{\tau} \int |\partial_t v|^2 e^{5\rh}\sqrt{\gamma}
\label{RRR}
\end{equation}
for  $\rh<-C$ and $\tau>C$. Together with (\ref{decom}), we conclude that
\begin{align}
I_1 &\geq C \tau \int |v_\rh|^2 \sqrt{\gamma}+ \frac{1}{\tau} \int |\partial_t v|^2 e^{5\rh}\sqrt{\gamma}+C\ep^2\tau \int |\nabla_\theta v|^2
e^{\ep \rh} \sqrt{\gamma} \nonumber \\
&-C \int \tau^2 e^{\ep \rh} v^2 \sqrt{\gamma}.
\label{III}
\end{align}

Next, we estimate the term $\|\mathcal{Q}_\tau^{1}(v)\|^2$ in (\ref{compren}). For some $\eta>0$, it is true that
\begin{align}
\int |\mathcal{Q}_\tau^{1}(v)|^2\sqrt{\gamma} &=\int \big( \mathcal{Q}_\tau^{1}(v)-\eta \tau v e^{\ep \rh} +\eta \tau v e^{\ep \rh}\big)^2\sqrt{\gamma} \nonumber\\
&\geq 2\eta \tau \int \big(\mathcal{Q}_\tau^{1}(v)-\eta \tau e^{\ep \rh}v\big) v e^{\ep \rh} \sqrt{\gamma}.\label{trivi}
\end{align}
Recall that $\mathcal{Q}_\tau^1$ in (\ref{back}) as
$$\mathcal{Q}_\tau^1 (v)=\partial_\rh^2 v+A_0(\rh) v-a(\rh) \tilde{V}(e^{f(\rh)}, \theta, t)v+ (1+\ep e^{\ep \rh})^2\triangle_{\theta} v.  $$
The inequality (\ref{trivi}) yields that
\begin{align}
\int |\mathcal{Q}_\tau^{1}(v)|^2\sqrt{\gamma}&\geq  2\eta \tau \int \{\partial_\rh^2 v+A_0(\rh) v-a(\rh) \tilde{V}(e^{f(\rh)}, \theta, t)v+ (1+\ep e^{\ep \rh})^2\triangle_{\theta} v \nonumber \\
&-\eta \tau e^{\ep \rh}v \} v e^{\ep \rh} \sqrt{\gamma}.
\label{essen}
\end{align}

To find a lower bound of $\|\mathcal{Q}_\tau^{1}(v)\|^2$, we estimate each term in the right hand side of (\ref{essen}). From integration by parts argument, it follows that
\begin{align}
2\eta \tau \int \partial_\rh^2 v v e^{\ep \rh} \sqrt{\gamma}&=-2\eta \tau \int |\partial_\rh v|^2 e^{\ep \rh} \sqrt{\gamma}
-2\eta \ep \tau \int \partial_\rh v v  e^{\ep \rh} \sqrt{\gamma}
-2\eta \tau \int \partial_\rh v v  e^{\ep \rh} \partial_\rh (\ln \sqrt{\gamma})\sqrt{\gamma}\nonumber \\
&\geq  -4\eta \tau \int |\partial_\rh v|^2 e^{\ep \rh} \sqrt{\gamma}-4\eta \tau \int  v^2 e^{\ep \rh} \sqrt{\gamma},
\label{Q1}
\end{align}
where we have used Cauchy-Scharwtz inequality and the estimate (\ref{often}).
By the definition of $A_0(\rh)$ in (\ref{A0A0}) and $a(\rh)$ in (\ref{aaa}), we have
\begin{align}
2\eta \tau  \int \big [A_0(\rh) v-a(\rh) \tilde{V}(e^{f(\rh)}, \theta, t)v -\eta \tau v e^{\ep \rh}\big] v  e^{\ep \rh} \sqrt{\gamma}
\geq 2\eta \tau \int (\tau^2- \|\tilde{V}\|_{C^{1,1}}) v^2  e^{\ep \rh} \sqrt{\gamma}.
\label{Q2}
\end{align}
The integration by parts argument shows that
\begin{equation}
2\eta \tau  \int (1+\ep e^{\ep \rh})^2\triangle_{\theta} v v  e^{\ep \rh} \sqrt{\gamma}\geq
-2\eta \tau  \int (1+\ep e^{\ep \rh})^2 |\nabla_\theta v|^2 e^{\ep \rh} \sqrt{\gamma}.
\label{Q3}
\end{equation}
Since we have assumed that $\tau >C (1+\|\tilde{V}\|_{C^{1,1}}^{\frac{1}{2}})$, from the inequalities (\ref{Q1})--(\ref{Q3}), we obtain that
\begin{align}
\int |\mathcal{Q}_\tau^{1}(v)|^2\sqrt{\gamma} & \geq \eta \tau^3 \int v^2 e^{\ep \rh} \sqrt{\gamma}- 4\eta \tau  \int |\partial_\rh v|^2 e^{\ep \rh} \sqrt{\gamma} \nonumber \\
&-2\eta \tau  \int (1+\ep e^{\ep \rh})^2 |\nabla_\theta v|^2 e^{\ep \rh} \sqrt{\gamma}.
\end{align}
Choosing the small $\eta$ such that $\eta=\frac{C\ep^2}{4}$, taking the estimates (\ref{III}) and the last inequality into account
 yields that
\begin{align}
\|\mathcal{Q}_\tau^{1}(v)\|^2+I_1&\geq  C\tau^3 \int  v^2 e^{\ep \rh} \sqrt{\gamma}+C\tau  \int |\partial_\rh v|^2 e^{\ep \rh} \sqrt{\gamma} \nonumber \\
&+C\tau  \int |\nabla_\theta v|^2 e^{\ep \rh} \sqrt{\gamma}+\frac{1}{\tau}\int |\partial_t v|^2 e^{5\rh}\sqrt{\gamma}.
\label{key}
\end{align}
To estimate $\|\mathcal{Q}_\tau(v)\|$ in (\ref{compren}), we are left with $I_2$ in (\ref{I2I2}). Our goal is to control $I_2$  by the right hand side of (\ref{key}).
As before, we estimate each term in the right hand side of $I_2$ in (\ref{I2I2}) by integration by parts argument. It is clear that
\begin{equation}
|\int A_1(\rh)|\partial_\rh v|^2 \partial_\rh (\ln \sqrt{\gamma})\sqrt{\gamma}|\leq C\tau  \int |\partial_\rh v|^2 e^{ \rh} \sqrt{\gamma},
\label{I21}
\end{equation}
since $ \partial_\rh (\ln \sqrt{\gamma})\leq C e^\rh$ and $0<\rh<-|\rh_0|$ with  $|\rh_0|$ sufficiently large.
By the Young's inequality,
\begin{equation}
2|\int a(\rh)\partial_t v \partial_\rh v \partial_\rh (\ln \sqrt{\gamma})\sqrt{\gamma} |\leq \frac{\delta}{\tau}\int |\partial_t v|^2 e^{5\rh}\sqrt{\gamma}+
C(\delta)
\tau  \int |\partial_\rh v|^2 e^{ \rh} \sqrt{\gamma}.
\end{equation}
If we choose $\delta$ to be small, since $\rh$  is sufficiently close to negative infinity, then
\begin{equation}
2|\int a(\rh)\partial_t v \partial_\rh v \partial_\rh (\ln \sqrt{\gamma})\sqrt{\gamma}|\leq  \frac{1}{2\tau}\int |\partial_t v|^2 e^{5\rh}\sqrt{\gamma}+
C
\tau  \int |\partial_\rh v|^2 e^{ \rh} \sqrt{\gamma}.
\label{I22}
\end{equation}
It is obvious that
\begin{equation}
2|\int \big(A_1(\rh) A_0(\rh)\big)v^2 \partial_\rh (\ln \sqrt{\gamma})\sqrt{\gamma}|\leq C\tau^3 \int v^2  e^{ \rh} \sqrt{\gamma}.
\label{I23}
\end{equation}
Since it is assumed that $\tau >C (1+\|\tilde{V}\|_{C^{1,1}}^{\frac{1}{2}})$, then
\begin{equation}
\tau (1+  \|\tilde{V}\|_{C^{1,1}})\int  v^2  e^{ \rh} \sqrt{\gamma} \leq C\tau^3 \int  v^2  e^{ \rh} \sqrt{\gamma}.
\label{I24}
\end{equation}
Together with the inequalities (\ref{I21})--(\ref{I24}), we derive that
\begin{align}
2|I_2| \leq &C\tau^3 \int v^2 e^{\ep \rh} \sqrt{\gamma}+C\tau  \int |\partial_\rh v|^2 e^{\ep \rh} \sqrt{\gamma} \nonumber \\
&+C\tau  \int |\nabla_\theta v|^2 e^{\ep \rh} \sqrt{\gamma}+\frac{1}{\tau}\int |\partial_t v|^2 e^{5\rh}\sqrt{\gamma}.
\label{come}
\end{align}
Hence, $I_2$ can be controlled above by the right hand side of (\ref{key}). Taking advantage of (\ref{compren}), (\ref{key}) and (\ref{come}) together, we obtain that
\begin{equation}
\int |\mathcal{Q}_\tau (v)|^2 \sqrt{\gamma}\geq  C\tau^3 \int v^2 e^{\ep \rh} \sqrt{\gamma}+C\tau  \int |\partial_\rh v|^2 e^{\ep \rh} \sqrt{\gamma}
+C\tau  \int |\nabla_\theta v|^2 e^{\ep \rh} \sqrt{\gamma}.
\label{keyy}
\end{equation}

At last, we deal with $\tilde{{\mathcal{Q}}}(v)$ in (\ref{below}). Since $\rh$ is close to negative infinity, for any fixed $0<
\ep<1$, it is easy to see that
\begin{align}
\int |\tilde{{\mathcal{Q}}}(v)|^2 \sqrt{\gamma} \leq C \tau^2 \int v^2 e^{ \rh} \sqrt{\gamma}
+C \int |\partial_\rh v|^2 e^{ \rh} \sqrt{\gamma},\label{rebelow}
\end{align}
which can be bounded  by the the right hand side of (\ref{keyy}).

Recall that in polar coordinates $(r, \theta)$ the volume element is
$r^{n-1}\sqrt{\gamma} dr d\theta$ and $\frac{1}{r}dr\approx d \rh$ as $\rh$ close to negative infinity.
From (\ref{basic}) and (\ref{keyy}),
 we have shown that, for any $u\in C^{\infty}_0\big (Q^T_{r_0}(x_0) \backslash \big\{\{ x_0\}\times(-T , \ T)\big\}\big)$  and $\tau > C (1+ \|\tilde{V}\|_{C^{1,1}}^{\frac{1}{2}})$,
\begin{align*}
&\int_{Q^T_{r_0}}\big(\triangle u-\partial_t u-\tilde{V}(r, \theta ,t)u\big)^2 e^{-2\tau g(r)} r^{4-n} d v_g dt \nonumber \\ &\geq C\int
(\tau r^2|\nabla u|^2 +\tau^3 u^2) e^{(-2\tau+\ep)g(r)} r^{-n} d v_g dt,
\end{align*}
where $g(r)=f^{-1}(\ln r)$. We arrive at the proof of Theorem \ref{thCarle}.

\end{proof}

\section{Three cylinder  inequalities}
The $L^2$ type three cylinder inequalities for parabolic equations have been established in e.g. \cite{EVe03}, \cite{V03} for the proof of the strong
unique continuation property.
In this section, we will derive the quantitative $L^\infty$ type three cylinder inequalities from  Carleman estimates in the last section. The norms of the coefficient functions in (\ref{goal}) or (\ref{goal2}) are explicitly characterized, which is crucial in showing the vanishing order.
The standard way is to apply those Carleman estimates to $u(x,t)\xi(x, t)$ where $\xi(x, t)$ is an appropriate cut-off function, $u(x, t)$ is a solution,  and then make an appropriate choice of the parameter $\tau$.  Recall that $r_0$ is the geodesic distance in the Carleman estimates in the last section, $0<\ep<1$ is some fixed constant and $T\in (-1, 1)$. We state the three cylinder inequality for parabolic equation (\ref{goal}) as follows.
\begin{lemma}
Let $0<3r_1<r_2<\frac{r_3}{2}<\frac{r_0}{4}$ and $u$ be a solution to \eqref{goal}. There exist a positive constant $C$ depending only on $\mathcal{M}$ and $\ep$ such that
\begin{align}
\|u\|_{L^\infty(Q^{T/2}_{r_2})}&\leq C  r_2^{-\frac{\ep}{2}}(\frac{r_0^2}{T}+1) M^{\frac{n+4}{4}}
\|u\|_{L^\infty(Q^{T}_{2r_1})}^{k_0} \|u\|_{L^\infty(Q^{T}_{r_3})}^{1-k_0} \nonumber \\ &+C(\frac{r_3}{r_2})^{\frac{n}{2}}M^{\frac{n+2}{4}} \exp \{ CM^\frac{1}{2} \big( g(\frac{r_3}{2})-g(r_1)\big)\} \|u\|_{L^\infty(Q^{{T}}_{2r_1})},
\label{infnity}
\end{align}
where $k_0=\frac{g(\frac{r_3}{2})-g(r_2)}{g(\frac{r_3}{2})-g(r_1)}.$
\label{lem2}
\end{lemma}

For the parabolic equation (\ref{goal2}) with bounded coefficient functions, we can establish the following  three cylinder inequality.
\begin{lemma}
Let $0<3r_1<r_2<\frac{r_3}{2}<\frac{r_0}{4}$ and $u$ be a solution to \eqref{goal2}. There exist a positive constant $C$ depending only on $\mathcal{M}$ and $\ep$ such that
\begin{align}
\|u\|_{L^\infty(Q^{T/2}_{r_2})}&\leq C  r_2^{-\frac{\ep}{2}}(\frac{r_0^2}{T}+1) (M_0^{\frac{1}{2}}+M_1)^{\frac{n+4}{2}}
\|u\|_{L^\infty(Q^{T}_{2r_1})}^{k_0} \|u\|_{L^\infty(Q^{T}_{r_3})}^{1-k_0} \nonumber \\ &+C(\frac{r_3}{r_2})^{\frac{n}{2}}(M_0^{\frac{1}{2}}+M_1)^{\frac{n+2}{2}} \exp \{ C(M_0^\frac{2}{3}+M_1^2) \big( g(\frac{r_3}{2})-g(r_1)\big)\} \|u\|_{L^\infty(Q^{{T}}_{2r_1})},
\label{infnity1}
\end{align}
where $k_0=\frac{g(\frac{r_3}{2})-g(r_2)}{g(\frac{r_3}{2})-g(r_1)}.$
\label{lem3}
\end{lemma}

With aid of the Carleman estimates (\ref{Carle}), we first show the proof of (\ref{infnity}).
\begin{proof}[Proof of Lemma \ref{lem2}]
Choose $0<3r_1<r_2<\frac{r_3}{2}<\frac{r_0}{4}$. We construct a smooth cut-off
function $\xi(x, t)=\psi(r)\varphi(t)$. We select $\psi(r)\in C^\infty_0(\mathbb B_{\frac{r_0}{2}})$ such that $\psi(r)=1$ in
$[\frac{3r_1}{2}, \ \frac{r_3}{2}]$ and $\psi(r)=0$ in
$[0, \ r_1]\cup [\frac{3r_3}{4}, \ r_3]$. Then $$|\nabla \psi|\leq \frac{C}{r_1} \quad \mbox{in} \ [r_1, \ \frac{3r_1}{2}] \quad \mbox{and} \quad |\nabla \psi|\leq \frac{C}{r_3} \quad \mbox{in} \ [\frac{r_3}{2}, \ \frac{3r_3}{4}].$$

We also select a cut-off function with respect to $t$ variable and adapt the arguments in \cite{V03}. Let $T_1=\frac{5T}{6}$ and $T_2=\frac{2T}{3}$.
Select $\varphi(t)$ be a even function such that
  $\varphi(t)\in C^\infty_0(-T, \ T)$.
  Set $\varphi(t)=1$ in $[-T_2, ,  \ T_2]$,  $\varphi(t)=0$ in $[-T,  \ -T_1]\cup [T_1, \ T]$.
  Define
  \begin{equation}
  \varphi(t)= \left\{ \begin{array}{lll}  \varphi_1(t) \quad \quad t\in (-T_1, \ -T_2), \nonumber \medskip \\
  \varphi_2(t) \quad \quad t\in (T_2, \ T_1), \nonumber
  \end{array}
      \right.
  \end{equation}
  where
$$ \varphi_1(t)=\exp \{ -\frac{ T^3(T_2+t)^4}{ (T_1+t)^3(T_1-T_2)^4}\} $$
and
$$ \varphi_2(t)=\exp \{ -\frac{ T^3(t-T_2)^4}{ (T_1-t)^3(T_1-T_2)^4}\}. $$
We can check that
$$|\varphi'(t)|\leq \frac{C}{T} \ \mbox{in} \  (-T_1, \ -T_2)\cup (T_2, \ T_1).$$
Now we define the following sets
\begin{align*}
&D_1^{'}=\{ (x, t)\in Q^T_{r_0} | \frac{3r_1}{2}<r<\frac{r_3}{2}, \quad  t\in [-T_1, \ -T_2] \}\\
&D_1^{''}=\{ (x, t)\in Q^T_{r_0} | \frac{3r_1}{2}<r<\frac{r_3}{2}, \quad  t\in [T_2, \ T_1] \} \\
&D_2=\{ (x, t)\in Q^T_{r_0}  | {r_1}<r<\frac{3r_1}{2}, \quad t\in [-T_1, \ T_1] \}, \\
&D_3=\{ (x, t)\in Q^T_{r_0}  | \frac{r_3}{2}<r<\frac{3r_3}{4}, \quad t\in [-T_1, \ T_1] \},\\
&D_4=\{ (x, t)\in Q^T_{r_0} | \frac{3r_1}{2}<r<\frac{r_3}{2}, \quad t\in [-T_2, \ T_2] \}, \\
&D_1= D_1^{'}\cup D_1^{''}.
\end{align*}
 Note that $\xi(x, t)\equiv 1$ on $D_4$, i.e. $u(x, t) \xi(x, t)\equiv u(x, t)$ on $D_4$. On $Q^T_{r_0} \backslash \cup_{i=1}^4 D_i$, $\xi(x, t)\equiv 0$. Then,
$u(x, t) \xi(x, t)\equiv 0$ on $Q^T_{r_0} \backslash \cup_{i=1}^4 D_i$. By the standard regularity argument,
Choosing $u(x, t)\xi(x, t)$ as the test function in the Carleman estimates (\ref{Carle}) yields that
\begin{align*}
\int &\big(\tau r^2|\nabla (u\xi) |^2+\tau^3 (u\xi) ^2\big) e^{(-2\tau+\ep)g(r)} r^{-n} d v_g dt
\\
&\leq C\int_{Q^T_{r_0}}\big(\triangle (u\xi )-\partial_t (u\xi )-\tilde{V}(r, \theta ,t)u\xi \big)^2 e^{-2\tau g(r)} r^{4-n} d v_g dt.
\end{align*}
It follows that
\begin{align}
&\tau^3 \int_{ D_4\cup D_1}  (u\xi)^2 e^{(-2\tau+\ep)g(r)} r^{-n} d v_g dt \nonumber \\
&\leq C  \int_{{\cup^4_{i=1}D_i}}\big(\triangle (u\xi )-\partial_t(u\xi )-\tilde{V}(r, \theta ,t)u\xi \big)^2 e^{-2\tau g(r)} r^{4-n} d v_g dt.
\label{lat1}
\end{align}
From the definition of $\xi(x, t)$, we obtain that
\begin{align}
&\tau^3 \int_{ D_4}  u^2 e^{(-2\tau+\ep)g(r)} r^{-n} d v_g dt \nonumber \\
&\leq C  \int_{{\cup^4_{i=2}D_i}}\big(\triangle (u\xi )-\partial_t(u\xi )-\tilde{V}(r, \theta ,t)u\xi \big)^2 e^{-2\tau g(r)} r^{4-n} d v_g dt+L_1,
\label{lat}
\end{align}
where
\begin{align*}L_1 &=  \int_{D_1}\big(\triangle (u\xi )-\partial_t(u\xi )-\tilde{V}(r, \theta ,t)u\xi \big)^2 e^{-2\tau g(r)} r^{4-n} d v_g dt \nonumber \\
&-\tau^3 \int_{ D_1}  (u\xi)^2 e^{(-2\tau+\ep)g(r)} r^{-n} d v_g dt.
\end{align*}
We investigate each integral in the right hand side of inequality (\ref{lat}). From the equation (\ref{goal}) itself,  on the domain $D_4$, we obtain that
\begin{align}
  &\int_{D_4}\big(\triangle (u\xi )-\partial_t(u\xi)-\tilde{V}(r, \theta ,t)u\xi \big)^2 e^{-2\tau g(r)} r^{4-n} d v_g dt \nonumber\\
  &=\int_{D_4}\big(\triangle \xi u+ 2\nabla \xi\cdot \nabla u +\xi\triangle u-\xi_t u-\xi u_t-\tilde{V}(r, \theta ,t)u\xi \big)^2 e^{-2\tau g(r)} r^{4-n} d v_g dt   \nonumber\\
  &=0,
\label{latter}
\end{align}
 since $\nabla \xi=0$ and $\xi_t=0$ in $D_4$. We exam the integral $L_1$. On the domain $D_1$, it holds that $\psi(r)=1$. Considering $u$ is the solution of the equation (\ref{goal}), it follows that
\begin{align}
  &\int_{D_1}\big(\triangle (u\xi )-\partial_t(u\xi )-\tilde{V}(r, \theta ,t)u\xi \big)^2 e^{-2\tau g(r)} r^{4-n} d v_g dt \nonumber\\
  &=\int_{D_1}\big(\triangle \xi u+ 2\nabla \xi\cdot \nabla u +\xi\triangle u-\xi_t u-\xi u_t-\tilde{V}(r, \theta ,t)u\xi \big)^2 e^{-2\tau g(r)} r^{4-n} d v_g dt   \nonumber\\
  &=\int_{D_1}\varphi_t^2 u^2 e^{-2\tau g(r)} r^{4-n} d v_g dt
\label{estiD1}
\end{align}

 From the inequality (\ref{estiD1}), it follows that
$$ L_1 =  \int_{D_1}E(r, t;\tau) u^2  e^{-2\tau g(r)} r^{-n} d v_g dt, $$
where
$$ E(r, t;\tau) = \varphi^2(\frac{\varphi_t^2}{\varphi^2} r^4-\tau^3 e^{\ep g(r)}). $$
We first work on the domain $D_1^{'}$ with $\varphi(t)=\varphi_1(t)$. Calculations show that
$$ \varphi_1^{'}(t)=\varphi_1(t)\frac{-T^3(T_2+t)^3(4T_1-3T_2+t)}{(T_1-T_2)^4(T_1+t)^4}. $$
Since $g(r)\approx \ln r$ as $r\to 0$,  we have
$$ E(r, t;\tau) \leq \tau^3 r^\ep \varphi_1^2( \frac{Cr^{4-\ep} T^6}{(T_1+t)^8 \tau^3}-\frac{1}{2}) $$
for some fixed $0<\ep<1$. Furthermore, we introduce the set
$$ D^{'}_{1, \tau}=\{ (x, t)\in  D^{'}_{1}| -\frac{1}{2}+ \frac{Cr^{4-\ep} T^6}{(T_1+t)^8 \tau^3}\geq 0\}. $$
In the region $D^{'}_{1, \tau}$,
\begin{equation} \tau^3\leq  \frac{Cr^{4-\ep} T^6}{(T_1+t)^8}.
\label{kubi}
\end{equation}
It is true that
 \begin{align}
 \int_{D_1^{'}}E(r, t;\tau) u^2  e^{-2\tau g(r)} r^{-n} d v_g dt \leq \frac{C}{T^2}\int_{D_{1,\tau}^{'}} \varphi_1 u^2
e^{-2\tau g(r)} r^{4-n}d v_g dt,
\label{twos}
\end{align}
where we have used the fact that
$$\sup_{[\frac{4}{5}, \ 1)}(1-s)^{-8} \exp\{ -(1-s)^{-3}(\frac{4}{5}-s)^4\} \leq C $$
and $\frac{T_2}{T_1}=\frac{4}{5}$. From (\ref{kubi}) at the region $D_{1,\tau}^{'}$, it follows that
\begin{equation}
\frac{T_1+t}{T}\leq (\frac{C r^{4-\ep}}{T^2\tau^3})^{\frac{1}{8}}.
\label{dodo} \end{equation}
Now we choose $\tau> (Cr_0^{4-\ep} T^{-2} 12^8)^{1/3}$ for some fixed $r_0$, then
\begin{equation} \frac{T_1+t}{T}\leq \frac{1}{12}. \label{dodo1}\end{equation}
Note that $T_1-T_2=\frac{T}{6}$. The inequalities (\ref{dodo}) and (\ref{dodo1}) implies that
\begin{equation} |T_2+t|\geq \frac{T_1-T_2}{2}.
\label{guosi}
\end{equation}
Recall the definition of $\varphi_1(t)$, (\ref{dodo}) and (\ref{guosi}), we obtain that
$$ \varphi_1 e^{-2\tau g(r)}r^{-n}\leq \exp\{ -\frac{1}{2^4}(\frac{\tau^3 T^2}{C r^{4-\ep}})^{3/8}-(2\tau+n)\ln r\}.$$
Thus, if $\tau>C$ for some large $C$ and $r$ is sufficiently small, the inequality
(\ref{twos}) and the last inequality imply that
\begin{equation}
 \int_{D_1^{'}}E(r, t;\tau) u^2  e^{-2\tau g(r)} r^{-n} d v_g dt \leq C (\frac{r_0^2}{T})^2 \int_{D_1^{'}}  u^2 d v_g dt.
 \label{fly}
\end{equation}
Arguing in the same way with the $\varphi(t)=\varphi_2(t)$ on the region $D_1^{''}$, we will get the similar estimates as (\ref{fly}). Therefore, we arrive at
$$
L_1 \leq C (\frac{r_0^2}{T})^2 \int_{D_1}  u^2 d v_g dt.
$$
It is also true that
\begin{equation}
L_1\leq C e^{-2\tau g(\frac{r_3}{2})} (\frac{r_0^2}{T}+1)^2 \int_{Q_{r_3}^T} u^2 d v_g dt.
\label{dudu}
\end{equation}
 On the domain $D_3$, by the fact that $-g(r)$ is decreasing, we have
\begin{align}
  &\int_{D_2}\big(\triangle (u\xi )-\partial_t (u\xi )-\tilde{V}(r, \theta ,t)u\xi \big)^2 e^{-2\tau g(r)} r^{4-n} d v_g dt \nonumber\\
  &=\int_{D_2}\big(\triangle \xi u+ 2\nabla \xi\cdot \nabla u +\xi\triangle u-\xi_t u-\xi u_t-\tilde{V}(r, \theta ,t)u\xi \big)^2 e^{-2\tau g(r)} r^{4-n} d v_g dt   \nonumber\\
  &\leq C\int_{D_2}(\frac{1}{r_1^4} u^2+\frac{1}{r_1^2}|\nabla u|^2+\frac{1}{T^2} u^2) e^{-2\tau g(r)} r^{4-n} d v_g dt \nonumber\\
 &\leq C e^{-2\tau g(r_1)} r_1^{4-n} \int_{D_2}(\frac{1}{r_1^4}u^2+\frac{1}{r_1^2}|\nabla u|^2+\frac{1}{T^2}u^2)  d v_g dt.
 \label{estiD2f}
\end{align}
Using the standard Caccioppoli inequality for parabolic equations (\ref{goal}), it follows that
\begin{equation}
\int_{D_2} |\nabla u|^2 d v_g dt\leq C(1+\|\tilde{V}\|_{L^\infty})(\frac{1}{r_1^2}+\frac{1}{T})
\int_{\mathbb B_{\frac{7r_1}{4}}\backslash \mathbb B_{\frac{3r_1}{4}}\times [\frac{-11T}{12},\ \frac{11T}{12}]} u^2d v_g dt.
\label{cacci}
\end{equation}
Thus, the inequalities (\ref{estiD2f}) and (\ref{cacci}) yield that
\begin{align}
&\int_{D_2}\big(\triangle (u\xi )-\partial_t (u\xi )-\tilde{V}(r, \theta ,t)u\xi \big)^2 e^{-2\tau g(r)} r^{4-n} d v_g dt \nonumber\\
&\leq C e^{-2\tau g(r_1)}r_1^{-n} (\frac{r_1^2}{T}+1)^2 (1+\|\tilde{V}\|_{L^\infty} ) \int_{\mathbb B_{\frac{7r_1}{4}}\backslash \mathbb B_{\frac{3r_1}{4}}\times [\frac{-11T}{12},\ \frac{11T}{12}]} u^2d v_g dt.
\label{estiD3}
\end{align}
We use the similar strategy to deal with the integral on the domain $D_3$. Using the assumption of $\xi$ and then the
Caccioppoli inequality as (\ref{estiD3}), we obtain that
\begin{align}
&\int_{D_3}\big(\triangle (u\xi )-\partial_t(u\xi )-\tilde{V}(r, \theta ,t)u\xi \big)^2 e^{-2\tau g(r)} r^{4-n} d v_g dt \nonumber\\
&\leq C e^{-2\tau g(\frac{r_3}{2})}r_3^{-n} (\frac{r_3^2}{T}+1)^2(1+ \|\tilde{V}\|_{L^\infty} ) \int_{\mathbb B_{\frac{7r_3}{8}}\backslash \mathbb B_{\frac{r_3}{3}}\times [\frac{-11T}{12},\ \frac{11T}{12}]}u^2d v_g dt.
\label{estiD2}
\end{align}
Define a new set
$$D_4^{r_2}=\{(x, t)\in D_4| \quad  |x|\leq r_2\}$$ for
${r_2}\leq \frac{r_3}{2}$.
By the fact that $-g(r)$ is a decreasing function again, it is clear that
\begin{align}
\tau^3 e^{-2\tau g(r_2)} r^{-n+\ep}_2 \int_{D_4^{r_2}}  u^2 d v_g dt
\leq C \tau^3  \int_{ D_4}  u^2 e^{(-2\tau+\ep)g(r)} r^{-n} d v_g dt.
\label{comeon}
\end{align}
Together with the inequalities (\ref{lat}), (\ref{latter}), (\ref{dudu}), (\ref{estiD3}), (\ref{estiD2}) and (\ref{comeon}), and the assumption that
$\|\tilde{V}\|_{C^{1, 1}}\leq M$,
it follows that
\begin{align}
 \int_{ D_4^{r_2}}  u^2 d v_g dt \leq C Mr^{-\ep}_2 (\frac{r_0^2}{T}+1)^2&\big[ \exp\{-2\tau\big(g(r_1)-g(r_2)\big)\} (\frac{r_2}{r_1})^n \int_{ Q^T_{2r_1}} u^2 dv_g dt \nonumber\\
 & + \exp\{-2\tau\big(g(\frac{r_3}{2})-g(r_2)\big)\} (\frac{r_2}{r_3})^n \int_{ Q^T_{r_3}} u^2 dv_g dt\big].
 \label{add}
\end{align}
We add  $\int_{\mathbb B_{\frac{3r_1}{2}}\times [-T_2, \ T_2]} u^2  dv_g dt$ to both sides of (\ref{add}). Recall that $T_2=\frac{2T}{3}$. Since $\exp\{-2\tau\big(g(r_1)-g(r_2)\big)\}>1$, it follows that
\begin{align}
 \int_{Q^{\frac{2T}{3}}_{r_2}}  |u|^2 d v_g dt \leq C Mr^{-\ep}_2 (\frac{r_0^2}{T}+1)^2&[ \exp\{-2\tau\big(g(r_1)-g(r_2)\big)\} (\frac{r_2}{r_1})^n \int_{ Q^T_{2r_1}} u^2 dv_g dt \nonumber\\
 & + \exp\{-2\tau\big(g(\frac{r_3}{2})-g(r_2)\big)\} (\frac{r_2}{r_3})^n \int_{ Q^T_{r_3}} u^2 dv_g dt].
 \label{rewrite}
\end{align}
For ease of notation, set
$$\beta_1=\big( Mr_2^{-\ep}(\frac{r_0^2}{T}+1)^2  (\frac{r_2}{r_1})^{n}\big)^{\frac{1}{2}}  $$and
$$\beta_1=\big( Mr_2^{-\ep}(\frac{r_0^2}{T}+1)^2 (\frac{r_2}{r_3})^{n}\big)^{\frac{1}{2}}.  $$
Let
$$\big(\int_{Q^T_{2r_1}} u^2  dv_g dt\big)^{\frac{1}{2}}=U_1  \quad \mbox{and} \quad \big(\int_{Q^T_{r_3}} u^2  dv_g dt\big)^{\frac{1}{2}}=U_2.
$$
Thus,  the inequality (\ref{rewrite}) can be rewritten as
\begin{equation}
\|u\|_{L^2(Q^{\frac{2T}{3}}_{r_2})}\leq C\beta_1  \exp\{-\tau\big(g(r_1)-g(r_2)\big)\} U_1+ C\beta_2 \exp\{-\tau\big(g(\frac{r_3}{2})-g(r_2)\big)\}U_2. \label{write}
\end{equation}
Define a new parameter $k_0$ as
$$ \frac{1}{k_0}=\frac{g(\frac{r_3}{2})-g(r_1)}{g(\frac{r_3}{2})-g(r_2)}. $$
Notice that $0<k_0<1$.
If $r_1$ is sufficiently small and $r_2$, $r_3$ are fixed constants, then $\frac{1}{k_0}\approx \ln \frac{1}{r_1}$. Set
$$\tau_1=\frac{k_0}{g(\frac{r_3}{2})-g(r_2)}\ln \frac{\beta_2 U_2}{ \beta_1 U_1}.    $$

On one hand, if $\tau_1> CM^{\frac{1}{2}}$, the previous calculations hold with such $\tau_1$. We replace those $\tau$ by such $\tau_1$. Thus, we get from (\ref{write}) that
\begin{equation}
\|u\|_{L^2(Q^{\frac{2T}{3}}_{r_2})}\leq 2C(\beta_1 U_1)^{k_0}(\beta_2 U_2)^{1-k_0}.
\end{equation}
That is,
\begin{equation}
\|u\|_{L^2(Q^{\frac{2T}{3}}_{r_2})}\leq 2M^{\frac{1}{2}} C r_2^{-\frac{\ep}{2}}(\frac{r_0^2}{T}+1)
\big[(\frac{r_2}{r_1})^{\frac{n}{2}}\|u\|_{L^2(Q^{{T}}_{2r_1})} \big]^{k_0} \big[(\frac{r_2}{r_3})^{\frac{n}{2}}\|u\|_{L^2(Q^{{T}}_{r_3})} \big]^{1-k_0}.
\label{ineq1}
\end{equation}
On the other hand, if $\tau_1\leq  CM^{\frac{1}{2}}$, it follows that
$$\beta_2 U_2\leq \exp\{CM^{\frac{1}{2}} \big( g(\frac{r_3}{2})-g(r_1)\big)\} \beta_1 U_1.  $$
We can write the last inequality as
\begin{equation}
\|u\|_{L^2(Q^{\frac{2T}{3}}_{r_2})}\leq (\frac{r_3}{r_1})^{\frac{n}{2}} \exp\{CM^{\frac{1}{2}} \big( g(\frac{r_3}{2})-g(r_1)\big)\} \|u\|_{L^2(Q^{{T}}_{2r_1})}.
\label{ineq2}
\end{equation}
Combining the inequalities (\ref{ineq1}) and (\ref{ineq2}), we derive the following $L^2$ version of three cylinder inequality,
\begin{align}
\|u\|_{L^2(Q^{\frac{2T}{3}}_{r_2})}&\leq C M^{\frac{1}{2}} r_2^{-\frac{\ep}{2}}(\frac{r_0^2}{T}+1)
\big[(\frac{r_2}{r_1})^{\frac{n}{2}}\|u\|_{L^2(Q^{T}_{2r_1})} \big]^{k_0} \big[(\frac{r_2}{r_3})^{\frac{n}{2}}\|u\|_{L^2(Q^{{T}}_{r_3})} \big]^{1-k_0} \nonumber \\ &+C(\frac{r_3}{r_1})^{\frac{n}{2}} \exp\{CM^{\frac{1}{2}} \big( g(\frac{r_3}{2})-g(r_1)\big)\} \|u\|_{L^2(Q^{{T}}_{2r_1})}.
\label{three}
\end{align}
For the parabolic equations (\ref{goal}), the following standard local $L^\infty$ estimates hold
\begin{equation}
 \|u\|_{L^\infty(Q^{{T/2}}_{R/2})}\leq CR^{-\frac{n}{2}} T^{-\frac{1}{2}}(1+\|\tilde{V}\|_{L^\infty }^{\frac{n+2}{4}})
 \|u\|_{L^2(Q^{{T}}_{R})}.
 \label{lllin}
\end{equation}
From (\ref{three}) and (\ref{lllin}), we have the $L^\infty$ version of three cylinder inequality,
\begin{align}
\|u\|_{L^\infty(Q^{{T}/{2}}_{r_2})}&\leq C  r_2^{-\frac{\ep}{2}}(\frac{r_0^2}{T}+1) M^{\frac{n+4}{4}}
\|u\|_{L^\infty(Q^{T}_{2r_1})}^{k_0} \|u\|_{L^\infty(Q^{T}_{r_3})}^{1-k_0} \nonumber \\ &+C(\frac{r_3}{r_2})^{\frac{n}{2}} M^{\frac{n+2}{4}}\exp\{C M^{\frac{1}{2}} \big( g(\frac{r_3}{2})-g(r_1)\big)\} \|u\|_{L^\infty(Q^{{T}}_{2r_1})}.
\label{Koinfnity}
\end{align}
This completes the proof the Lemma \ref{lem2}.

\end{proof}
The proof of the three cylinder inequality in Lemma \ref{lem3} is very similar to that in Lemma \ref{lem2}. We only sketch the proof.
\begin{proof}[Proof of Lemma \ref{lem3}]
We apply the same test function $u(x,t)\xi(x, t)$ in the Carleman estimates (\ref{Carle2}). From the Caccioppoli inequality for the equation (\ref{goal2}), it holds that
\begin{equation}\int_{D_2} |\nabla u|^2 d v_g dt\leq C(1+\|{V}\|_{L^\infty}+\|{W}\|_{L^\infty}^2)(\frac{1}{r_1^2}+\frac{1}{T})
\int_{\mathbb B_{\frac{7r_1}{4}}\backslash \mathbb B_{\frac{3r_1}{4}}\times [\frac{-11T}{12},\ \frac{11T}{12}]} u^2d v_g dt.
\label{cacciop}
\end{equation}

Performing the discussions as Lemma \ref{lem2}  in two cases as $\tau_1> C(M_0^{\frac{2}{3}}+M_1^2)$ or $\tau_1< C(M_0^{\frac{2}{3}}+M_1^2)$, we derive the following $L^2$ type of three cylinder inequality,
\begin{align}
\|u\|_{L^2(Q^{\frac{2T}{3}}_{r_2})}&\leq C  (\frac{r_0^2}{T}+1) r_2^{-\frac{\ep}{2}}(M_0^{\frac{1}{2}}+M_1)^{\frac{n+2}{2}}
\big[(\frac{r_2}{r_1})^{\frac{n}{2}}\|u\|_{L^2(Q^{T}_{2r_1})} \big]^{k_0} \big[(\frac{r_2}{r_3})^{\frac{n}{2}}\|u\|_{L^2(Q^{{T}}_{r_3})} \big]^{1-k_0} \nonumber \\ &+C(\frac{r_0^2}{T}+1) (\frac{r_3}{r_1})^{\frac{n}{2}} \exp\{C(M_0^{\frac{2}{3}}+M_1^2) \big( g(\frac{r_3}{2})-g(r_1)\big)\} \|u\|_{L^2(Q^{{T}}_{2r_1})}.
\label{dcdc}
\end{align}
We also have the standard local $L^\infty$ estimates for parabolic equations (\ref{goal2}),
\begin{equation}
 \|u\|_{L^\infty(Q^{{T/2}}_{R/2})}\leq CR^{-\frac{n}{2}} T^{-\frac{1}{2}}(1+\|{V}\|_{L^\infty }^{\frac{1}{2}}+\|{W}\|_{L^\infty })^{\frac{n+2}{2}}
 \|u\|_{L^2(Q^{{T}}_{R})}.
 \label{lllin2}
\end{equation}
The combination of the inequalities (\ref{cacciop}), (\ref{dcdc}) and (\ref{lllin2}) will lead to (\ref{infnity1}). Thus, we arrive at the proof of lemma.

\end{proof}

\section{Propagation of smallness}
In this section, we will use the three cylinder inequality in the propagation
of smallness argument to establish the vanishing order for solutions on
$\mathcal{M}$. The propagation of smallness argument for elliptic equations based on the three-ball theorem has been performed in e.g. \cite{DF88}
\cite{Zhu16} for quantitative unique continuation. For parabolic equations, we adapt this idea with three cylinder inequalities to
obtain the order of vanishing estimate.

\begin{proof} [Proof of Theorem \ref{th1} ]
We start the propagation at any point $x_0\in \mathcal{M}$.
Choose $T=\frac{1}{2}$. Let $r_1=\frac{r}{2}$, $r_2=2r$ and $r_3=5r$.
We apply the $L^\infty$ version of three cylinder inequality (\ref{infnity}) for the equation (\ref{goal}), then

\begin{align}
\|u\|_{L^\infty (Q^{1/4}_{2r}) }&\leq C  r^{-\frac{\ep}{2}} M^{\frac{n+4}{4}}
\|u\|_{L^\infty (Q^{1/2}_{r})}^{k_0} \|u\|_{L^\infty(Q^{1/2}_{5r})}^{1-k_0} \nonumber \\ &+C M^{\frac{n+2}{4}}\exp\{C M^{\frac{1}{2}} \big( g(\frac{5r}{2})-g(\frac{r}{2})\big)\}
 \|u\|_{L^\infty (Q^{1/2}_{r})},
\label{innf}
\end{align}
where
$${k_0}=\frac{g(\frac{5r}{2})-g(2r)}{g(\frac{5r}{2})-g(\frac{r}{2})} $$
and $C$ depends on $\ep$ and the manifold $\mathcal{M}$.
Meanwhile, we can check that
$$c\leq g(\frac{5r}{2})-g(2r)\leq C \quad \mbox{and} \quad c\leq g(\frac{5r}{2})-g(\frac{r}{2})\leq C,$$
where $C$ and $c$ are positive constants are independent of $r$.
Thus, the parameter $k_0$ does not depend on  $r$.

We choose a small $r < \frac{r_0}{20}$ such that
$$\sup_{Q^{{1}/{2}}_r}|u|=\delta.$$
We claim that $\delta>0$. Otherwise, by the unique continuation property, $u\equiv 0$
in $\mathcal{M}^1$, which is obviously impossible. Since $\disp \sup_{\mathcal{M}}|u(x)|\geq 1$, by the continuity, there exists some $\bar x\in \mathcal{M}$ such that $$\disp \abs{u(\bar x)}=\sup_{\mathcal{M}}|u(x)|\geq 1.$$
There also exists a sequence of balls with radius $r$, centered at
$x_0, \ x_1, \ldots, x_m$ so that $x_{i+1}\in \mathbb B_{r}(x_i)$
for every $i=0, 1, \ldots, m$, and $\bar x\in \mathbb B_{r}(x_m)$. The number of
balls, $m$, depends on the radius $r$ that will be fixed later. The
application of $L^\infty$ version of three cylinder inequality
(\ref{innf}) at the $x_0$ and the boundedness assumption that
$\|u\|_{L^\infty(\mathcal{M}^1)}\leq {C_0}$ yield that
\begin{align}
\|u\|_{L^\infty(Q^{{1}/{4}}_{2r})}&\leq C  r^{-\frac{\ep}{2}} M^{\frac{n+4}{4}}
\delta^{k_0} C_0^{1-k_0}  +C \delta \exp\{CM^{\frac{1}{2}}\}.
\label{mama}
\end{align}

Choosing $T=\frac{1}{4}$, we apply the $L^\infty$ version of three cylinder theorem centered at $(x_1, 0)$. It follows that
\begin{align}
\|u\|_{L^\infty (Q^{1/8}_{2r}(x_1)) }&\leq C  r^{-\frac{\ep}{2}} M^{\frac{n+4}{4}}
\|u\|_{L^\infty (Q^{{1}/{4}}_{r}(x_1))}^{k_0} \|u\|_{L^\infty (Q^{{1}/{4}}_{5r}(x_1))}^{1-k_0} \nonumber \\ &+C \exp\{CM^{\frac{1}{2}} \big( g(\frac{5r}{2})-g(\frac{r}{2})\big)\}
 \|u\|_{L^\infty(Q^{{1}/{4}}_{r}(x_1 ))}.
\label{anotht}
\end{align}
Recall that $Q^{{1}/{4}}_{r}(x_1)=\mathbb B_r(x_1)\times (-\frac{1}{4}, \ \frac{1}{4})$. Since $x_1\in \mathbb B_r(x_0)$, then
\begin{equation}
\|u\|_{L^\infty (Q^{1/4}_{r}(x_1))}\leq  \|u\|_{L^\infty (Q^{1/4}_{2r}(x_0) )}.
\label{dumb}
\end{equation}
Therefore, from (\ref{mama}), (\ref{anotht}) and (\ref{dumb}),
\begin{align*}
\|u\|_{L^\infty(Q^{1/8}_{2r}(x_1)) }&\leq C r^{-\frac{\ep}{2}}M^{\frac{n+4}{4}}
[r^{-\frac{\ep}{2}} M^{\frac{n+4}{4}}\delta^{k_0} C_0^{1-k_0}]^{k_0} C_0^{1-k_0} \nonumber \\
&+ 2r^{-\frac{\ep}{2}} M^{\frac{n+4}{4}}\delta^{k_0} \exp\{CM^{\frac{1}{2}}\} C_0^{1-k_0}+\delta \exp\{CM^{\frac{1}{2}}\}.
\end{align*}

Repeating the above argument with a chain of cylinders centered at $(x_i, 0)$,
it follows that
\begin{equation}
\|u\|_{L^\infty (Q^{1/2(i+2)}_{r}(x_i))} \leq  C_i\exp\{D_i M^{\frac{1}{2}}\} r^{-E_i\ep} \delta^{F_i}
\label{itera}
\end{equation}
for $i=0, 1, \cdots, m$, where $C_i$ $D_i$, $E_i$, $F_i$ are constant depending on $m$,
$C_0$, $\ep$ and the manifold $\mathcal{M}$.

 After the finite $m$ steps, we will get that $\bar x\in \mathbb B_r( x_m )$. Since $u(\bar x)\geq 1$, it is clear that
 $$ \|u\|_{L^\infty(Q^{1/2(m+2)}_{r}(x_m))}\geq 1.  $$
Thus, from (\ref{itera}), we derive that
\begin{equation}
\sup_{Q^{{1}/{2}}_r}|u|=\delta\geq C r^C \exp\{-CM^{\frac{1}{2}}\}.
\end{equation}

Now we fix $r$ as a small number so that $m$ is a fixed
constant. We are going to apply the three cylinder inequality again. Choose $T=1$. Let $r_2= r$ and $r_3=4r$. Set
$2r_1 << r$, i.e. $r_1$ sufficiently small compared with $r$.
Applying the three cylinder
inequality (\ref{infnity}) at $(x_0, 0)$ implies that
$$ \delta \leq {I}_1 +I_2,$$
where$$
{ I}_1 = C  r^{-\frac{\ep}{2}} M^{\frac{n+4}{4}}
\|u\|_{L^\infty (Q^{1}_{2r_1})}^{k_0} \|u\|_{L^\infty (Q^{1}_{4r} )}^{1-k_0}$$
and
$$
I_2 = C M^{\frac{n+2}{4}}\exp\{CM^{\frac{1}{2}} \big( g(2r)-g(r_1)\big)\}
 \|u\|_{L^\infty (Q^{1}_{2r_1})}
$$
with $\disp k_0 = \frac{g(2r)-g(r)}{g(2r)-g(r_1)}$.

On one hand, if ${I}_1 \leq I_2$, then
\begin{align*}
&r^C \exp\{-CM^{\frac{1}{2}}\}
\le \delta \le 2 I_2 \\
&\le 2 C M^{\frac{n+2}{4}}\exp\{CM^{\frac{1}{2}} \big( g(2r)-g(r_1)\big)\}
 \|u\|_{L^\infty (Q^{1}_{2r_1})}.
\end{align*}
Recall that $g(r)\approx \ln r$ as $r\to 0$. Since $r_1 << r$, it is true that $g\pr{r_1}-\pr{C + g \pr{2r}}
\ge c g\pr{r_1}$ for some fixed constant $c>1$.  We get that
\begin{align}
\|u\|_{L^\iny(Q^1_{2r_1})} & \geq C \exp\{ CM^{\frac{1}{2}} g(r_1)\} \nonumber\\
 &=C r_1^{C{M}^{\frac{1}{2}}}. \label{case1}
\end{align}
On the other hand, if $I_2 \leq {I}_1$, we obtain that
\begin{align*}
&r^C \exp\{-CM^{\frac{1}{2}}\}
\le \delta \le 2 {I}_1 \\
&\le 2 C  r^{-\frac{\ep}{2}}M^{\frac{n+4}{4}}
\|u\|_{L^\infty (Q^{1}_{2r_1} )}^{k_0} \|u\|_{L^\infty (Q^{1}_{4r} )}^{1-k_0}.
\end{align*}
Taking $\|u\|_{L^\infty (\mathcal{M}^1)}\leq C_0$ into
consideration, it follows that
\begin{align*}
r^C \exp\{-C M^{\frac{1}{2}}\} M^{-\frac{n+4}{4}}
 &\leq  C \|u\|_{L^\infty (Q^{1}_{2r_1} )}^{k_0}.
\end{align*}
Since $r_1$ is
sufficiently small compared with $r$,  raising both sides to order $\frac{1}{k_0}$ in the last inequality and
taking the assumption that $\frac{1}{k_0}\approx \ln \frac{1}{r_1}$ into account, we obtain that
\begin{align}
\|u\|_{L^\infty(Q^{1}_{2r_1})}&\geq [Cr^C \exp\{-CM^{\frac{1}{2}}\}]^{\frac{1}{k_0}} \nonumber \\
& \geq C r_1^{CM^{\frac{1}{2}}}.
\label{case2}
\end{align}
Together with (\ref{case1}) and (\ref{case2}), we arrive the proof of Theorem \ref{th1}.
\end{proof}

At last, using the same idea of propagation of smallness argument, we show the proof of Theorem \ref{th2}. We only sketch the proof.
\begin{proof}[Proof of Theorem \ref{th2}]
To obtain the vanishing order for the parabolic equations (\ref{goal2}), we carry out the same propagation of smallness argument as the proof of Theorem \ref{th2}. We used the three cylinder inequality (\ref{infnity1}) for (\ref{goal2}). Observe from the proof of Theorem \ref{th1}, we can see that the power $C(M_0^{\frac{2}{3}}+M_1^2)$ in the exponential function $ \exp\{C(M_0^{\frac{2}{3}}+M_1^2)\}$ in (\ref{infnity1}) will  determine the rate of vanishing. Thus, from (\ref{infnity1}), the vanishing order for the solution of (\ref{goal2}) is given by $C(M_0^{\frac{2}{3}}+M_1^2)$. This completes the proof of Theorem \ref{th2}.
\end{proof}
Concerning about the assumption of solutions, we have the following remarks.
\begin{remark}
From the proof of Theorem \ref{th1} and \ref{th2}, we can define the vanishing order of solution
at $x_0\in \mathcal{M}$ by
\begin{equation}
\sup\{ k| \ \ \limsup_{r\to 0^+}\frac{\sup_{Q^{T_0}_r(x_0)} |u|}{r^k}\}
\label{defi2}
\end{equation}
for any fixed constant $0<T_0<1$, since the Carleman estimates in section 1 hold in any time interval containing $\{0\}$.
\label{rem1}
\end{remark}

\begin{remark}
We consider the  normalization of solutions in (\ref{normalize}). We can also normalize the solutions as follows
\begin{equation}
\|u(x,t)\|_{L^\infty(\mathcal{M})}\geq 1 \quad  \mbox{and}  \quad   \|u(x,t)\|_{L^2(\mathcal{M}^1)}\leq C_0
\label{normalize2}
\end{equation}
for some fixed constant $C_0$. The same vanishing order results in Theorem  \ref{th1} and \ref{th2} hold. For example, let's consider Theorem \ref{th1}.
By the local $L^\infty$ estimates, from (\ref{normalize2}), we can show that
\begin{equation}
\|u(x,t)\|_{L^\infty(\mathcal{M}^{\frac{1}{2}})}\leq C M^{\frac{n+2}{4}} C_0
\label{note}
\end{equation}
for some $C$ depending on $\mathcal{M}$. Observe from the propagation of smallness argument, for large constant $M$, the upper bound of the solution (\ref{note}) will be incorporated in the exponential function $\exp\{CM^{\frac{1}{2}}\}$ which determines the rate of vanishing.
\label{rem2}
\end{remark}

\begin{remark}
Using the same method, we are also able to deal with parabolic equations with Lipschitz leading coefficient
$$  \partial_i(a_{ij}(x, t)\partial_j u)-\partial_tu -W(x,t)\cdot \nabla u-V(x,t)u=0,$$
where
$$ C^{-1}|\xi|^2 \leq a^{ij}(x,t)\xi_i \xi_j\leq C|\xi|^2$$
and
$$ \sum^n_{i,j=1}|a^{ij}(x, t)-a^{ij}(y,s)|\leq C(|x-y|+|t-s|) $$
in $D\times (-T, T)$, where $D$ is a domain in $\mathbb R^n$.
\end{remark}


\begin{thebibliography}{CL}

\bibitem[AN08]{AN08} Giovanni Alessandrini and Luis Escauriaza. Null-contollability of one-dimensional parabolic equations. \emph{ESCAIM Control Optim.
Calc. Var.}, 12{2}: 284-293, 2008.

\bibitem[Bak13]{Bak13} Laurent Bakri. Carleman estimates for the {S}chr{\"o}dinger Operator.
Application to quantitative uniqueness. {\em Communication in
Partial Differential Equations,} 38(1):69-91, 2013.

\bibitem[Bak12]{Bak12}
Laurent Bakri.
\newblock Quantitative uniqueness for {S}chr{\"o}dinger operator.
\newblock {\em Indiana Univ. Math. J.}, 61(4):1565--1580, 2012.

\bibitem[BK05]{BK05}
Jean Bourgain and Carlos~E. Kenig.
\newblock On localization in the continuous {A}nderson-{B}ernoulli model in
  higher dimension.
\newblock {\em Invent. Math.}, 161(2):389--426, 2005.

\bibitem[C96]{C96} Xu-Yan Chen. A strong unique continuation theorem for parabolic equations. \emph{Math. Ann}.311: 603-630, 1996.

\bibitem[CK17]{CK17} Guher Camliyurt and Igor Kukaciva.  Quantitative unique continuation for a parabolic equation. to appear in \emph{Indiana Univ. Math. J.
}
\bibitem[LO74]{LO74} E.M. Landis and O.A. Oleinik.  Generalized analyticity and some related properties of solutions of elliptic and parabolic equations. \emph{Russian Math. Surv. } 29: 195-212, 1974.

\bibitem[Dav14]{Dav14}
Blair Davey.
\newblock Some quantitative unique continuation results for eigenfunctions of
  the magnetic {S}chr\"odinger operator. \newblock{\em Comm. Partial Differential Equations}, 39(5):876--945, 2014.


\bibitem[DF88]{DF88}
Harold Donnelly and Charles Fefferman.
\newblock Nodal sets of eigenfunctions on {R}iemannian manifolds.
\newblock {\em Invent. Math.}, 93(1):161--183, 1988.

\bibitem[DZ17]{DZ17} Blair Davey and Jiuyi Zhu. Quantitative uniqueness of solutions
to second order elliptic equations with
singular lower order terms,  arXiv:1702.04742.

\bibitem[DaZ17]{DaZ17} Blair Davey and Jiuyi Zhu. Quantitative uniqueness of solutions to second order elliptic equations with
singular potentials in two dimensions,  arXiv:1704.00632.

\bibitem[E00]{E00} Luis Escauriaza. Carleman inequalities and the heat operator.\emph{ Duke Math. J.} 104(1): 113-127, 2000.

\bibitem[EF03]{EF03} Luis Escauriaza and Francisco Javier Fern\'andez. Unique continuation for parabolic operators. \emph{Ark. Mat.} 41(1): 35-60, 2003.

\bibitem[EV01]{EV01} Luis Escauriaza and Luis Vega. Carleman inequalities and the heat opeator, II. \emph{Indiana Univ. Math. J.} 50(3): 1149-1169, 2001.

\bibitem[EVe03]{EVe03}Luis Escauriaza and Sergio Vessella. Optimal three cylinder inequalities for solutions to parabolic equations with Lipschitz leading coefficients. Inverse problems: theory and applications (Cortona/Pisa, 2002), 79-87, Contemp. Math., 333, Amer. Math. Soc., Providence, RI, 2003.

\bibitem[F03]{F03} Francisco Javier Fern\'andez. unique continuation for parabolic operators II. \emph{Comm. Partial Differential Equations} 28(9-10): 1597-1604, 2003.


\bibitem[Ken07]{Ken07}
Carlos~E. Kenig.
\newblock Some recent applications of unique continuation.
\newblock In {\em Recent developments in nonlinear partial differential
  equations}, volume 439 of {\em Contemp. Math.}, pages 25--56. Amer. Math.
  Soc., Providence, RI, 2007.

\bibitem[KSW15]{KSW15}
Carlos~E. Kenig, Luis Silvestre, and Jenn-Nan Wang.
\newblock On {L}andis' {C}onjecture in the {P}lane.
\newblock {\em Comm. Partial Differential Equations}, 40(4):766--789, 2015.

\bibitem[Ku98]{Ku98}
Igor Kukavica.
\newblock Quantitative uniqueness for second-order elliptic operators.
\newblock {\em Duke Math. J.}, 91(2):225--240, 1998.

\bibitem[L88]{L88} Fang-Hua Lin. A uniqueness theorem for parabolic equations. \emph{Comm. Pure Appl. Math.} 42: 125-136, 1988.

\bibitem[Mes92]{Mes92}
V.~Z. Meshkov.
\newblock On the possible rate of decay at infinity of solutions of second
  order partial differential equations.
\newblock {\em Math USSR SB.}, 72:343--361, 1992.

\bibitem[KT09]{KT09}Herbert Koch and Daniel Tataru. Carleman estimates and unique continuation for second order parabolic equaitons with nonsmooth
coefficients.\emph{ Comm. Partial Differential Equations} 34(4-6): 305-366, 2009.

\bibitem[P96]{P96} Chi-Cheung Poon. Unique continuation for parabolic equations. \emph{Comm. Partial Differential Equations} 21(3-4): 521-539, 1996.

\bibitem[S90]{S90} Christopher D. Sogge. A unique continuation theorem for second order parabolic differential operators, \emph{Ark. Mat.} 28:159-182, 1990.

\bibitem[V02]{V02} Sergio Vessella. Three cylinder inequalities and uniuqe continuation properties for parabolic equations. \emph{Rend Mat Acc Lincei Ser 9}, 13(2): 107-120, 2002.

\bibitem[V03]{V03} Sergio Vessella. Carleman estimates, optimal three cylinder inequality, and unique continuation properties for solutions to parabolic equations. \emph{Comm. Partial Differential Equations} 28(3-4): 637-676, 2003.
\bibitem[Zhu16]{Zhu16}
Jiuyi Zhu.
\newblock Quantitative uniqueness of elliptic equations.
\newblock {\em Amer. J. Math.}, 138(3):733--762, 2016.

\end{thebibliography}
\end{document}